\documentclass[12
pt]{amsproc}

\usepackage[english]{babel}

\RequirePackage{srcltx}

\usepackage{amsmath}
\usepackage{tikz}
\usepackage{verbatim}

\usepackage{amssymb}
\usepackage{amsxtra}

\usepackage{latexsym}
\usepackage{ifthen}
\usepackage{mathrsfs}

\usepackage{amsmath}
\usepackage{amssymb}
\usepackage{amsthm}

\newcommand{\T}{\mathbb{{T}}}
\newcommand{\N}{\mathbb{{N}}}

\def\N{{{\Bbb N}}}
\def\Z{{{\Bbb Z}}}
\def\T{{{\Bbb T}}}
\def\R{{\Bbb R}}
\def\C{{\Bbb C}}
\def\l{{\lambda }}
\def\a{{\alpha }}
\def\D{{\Delta }}

\def\a{{\alpha}}
\def\b{{\beta}}
\def\d{{\delta}}
\def\e{{\varepsilon}}
\def\s{{\sigma}}
\def\vp{{\varphi}}
\def\t{{\theta }}
\def\g{{\gamma }}

\def\w{{\omega }}

\def\E{\mathcal{E}}
\def\Sp{\mathcal{S}}

\def\){\right)}
\def\({\left(}

\def\supp{\operatorname{supp}}

\numberwithin{equation}{section}
\newtheorem{proposition}{Proposition}[section]

\newtheorem*{example}{Example}
\newtheorem{lemma}{Lemma}[section]
\newtheorem{remark}{Remark}[section]
\newtheorem{assumption}{Assumption}[section]
\newtheorem{theorem}{Theorem}[section]
\newtheorem{corollary}{Corollary}[section]

\numberwithin{equation}{section}

\def\la{\langle}
\def\ra{\rangle}

\def\wh{\widehat}

\def\SD{\mathcal{D}}

\def\q{\quad}

\def\la{\langle}
\def\ra{\rangle}

\def\wh{\widehat}

\def\SD{\mathcal{D}}

\def\supp{\text{\rm supp}}

\def\elra{\hbox to 2in{\rightarrowfill}}
\def\ella{\hbox to 2in{\leftarrowfill}}
\def\hrf{\hbox to 2in{\hrulefill}}
\def\hdotfill{\leaders\hbox to 1em{\hss .\hss}\hfill}

\def\N{{{\Bbb N}}}
\def\Z{{{\Bbb Z}}}
\def\T{{{\Bbb T}}}
\def\R{{\Bbb R}}
\def\C{{\Bbb C}}
\def\l{{\lambda }}
\def\a{{\alpha }}
\def\D{{\Delta }}

\def\a{{\alpha}}
\def\b{{\beta}}
\def\d{{\delta}}
\def\e{{\varepsilon}}
\def\s{{\sigma}}
\def\vp{{\varphi}}
\def\t{{\theta }}
\def\g{{\gamma }}

\def\w{{\omega }}

\def\){\right)}
\def\({\left(}

\def\supp{\operatorname{supp}}

\numberwithin{equation}{section}
\def\Dl{\bar{{\Delta}}}

\def\ab{{(a,b)}}

\def\R{\Bbb R}

\def\XXint#1#2#3{{\setbox0=\hbox{$#1{#2#3}{\int}$}
     \vcenter{\hbox{$#2#3$}}\kern-.5\wd0}}

\par

\bigskip
\bigskip

\begin{document}

\title[Smoothness of functions vs. smoothness of approximation processes]{
Smoothness of functions vs. smoothness of approximation processes
}

\author{Yu.~S.~Kolomoitsev$^*$}
\address{
Yu.
Kolomoitsev,
Universit\"at zu L\"ubeck,
Institut f\"ur Mathematik,
Ratzeburger Allee 160,
23562 L\"ubeck, Germany; Institute of Applied Mathematics and Mechanics of NAS of Ukraine, Generala Batyuka str. 19,  84116 Slov'yans'k, Ukraine
}
\email{kolomoitsev@math.uni-luebeck.de}

%
%

\author{S.~Yu.~Tikhonov}
\address{S.~Tikhonov,
Centre de Recerca Matem\`atica, Campus de Bellaterra, Edifici C 08193
Bellaterra, Barcelona, Spain; ICREA, Pg. Llu\'is Companys 23, 08010 Barcelona,
Spain, and Universitat Aut\'onoma de Barcelona}
\email{stikhonov@crm.cat}

\date{\today}

\thanks{$^*$Corresponding author}

\bigskip\bigskip\bigskip\bigskip

\bigskip
\begin{abstract}

We provide a comprehensive study of interrelations between different  measures of smoothness of
 functions on various domains and smoothness
properties of approximation processes.
Two general approaches to this problem have been developed: the first
based on geometric properties of Banach spaces and the second on  Littlewood-Paley
and H\"{o}rmander type multiplier theorems.
In particular, we obtain new sharp
inequalities for  measures of smoothness given by the $K$-functionals or moduli of smoothness.
 As examples of approximation processes we consider
best polynomial and spline approximations,
 Fourier multiplier operators on $\T^d$, $\R^d$, $[-1, 1]$, nonlinear
wavelet approximation, etc.

 \end{abstract}

\maketitle


\bigskip

\bigskip
\tableofcontents

\vskip 0.5cm

%
%
%
\newpage

\section{Introduction}

The fundamental problem in approximation theory is to find for a complicated function $f$ in a quasinormed space $X$ a close-by, simple approximant $P_n$
from a subset of $X$ such that the error of approximation  $\|f-P_n\|_X$ can be controlled by a specific majorant.
In many cases, this problem is solved completely and necessary and sufficient conditions are given in terms of smoothness properties  of either the function $f$ or approximants $P_n$ of $f$.

We illustrate this by considering the well-known case of approximation of periodic functions by trigonometric polynomials on $\T=[0,2\pi]$.
If
$f\in L_p(\T)$, $1\le p\le \infty$, and  $0<\alpha<r$,
for 
  the best approximant $T^*_n$ and the modulus of smoothness $\omega_r(f, t)_p$,
 the following conditions are equivalent:
\begin{eqnarray*}
&& (i_1)\qquad\quad \|f-T^*_n \|_p= \mathcal{O}(n^{-\alpha}),\qquad\qquad\\
&&(i_2) \qquad\quad\omega_r(f, t)_p= \mathcal{O}(t^\alpha),\qquad\qquad\\
&&(i_3) \qquad\quad \|(T^*_n)^{(r)} \|_p= \mathcal{O}(n^{r-\alpha}).
\end{eqnarray*}
See~\cite{St}, \cite{bu0}, and~\cite[Ch.~7]{DeLo}; for functions on $\T^d$ see \cite{Jo}.
Let us also mention earlier results by
Salem and Zygmund~\cite{SZ},  Zamansky~\cite{Za},
and  Civin~\cite{Ci}.
Similar results in the case of approximation by algebraic polynomials of functions on $[-1,1]$ can be found in~\cite[Ch.~8]{book} and~\cite{bu1}.


Equivalence $(i_1) \Leftrightarrow (i_2)$ easily follows from the classical Jackson and Bernstein approximation theorems, see, e.g.,~\cite[Ch.~7]{DeLo}, given by
\begin{equation*}\label{-JacksonSO}
E_n(f)_p\lesssim \w_r\(f,1/n\)_p\lesssim \frac1{n^r}\sum_{k=0}^n (k+1)^{r-1}E_k(f)_p, \quad 1\le p\le \infty,
\end{equation*}
or their sharper versions for $1<p<\infty$, see, e.g.,~\cite{DDT},
\begin{equation*}\label{DDT}
  \frac1{n^r}\bigg(\sum_{k=0}^n (k+1)^{r \tau-1}E_k(f)_p^{\tau}\bigg)^{\frac1\tau}
\lesssim \w_r\(f,1/n\)_p\lesssim \frac1{n^r}\bigg(\sum_{k=0}^n (k+1)^{r \t-1}E_k(f)_p^{\t}\bigg)^{\frac1\t},
\end{equation*}
where $E_n(f)_p$ is the error of the best approximation,
$\tau=\max(p,2)$ and $\t=\min(p,2)$.

The equivalence $(i_2) \Leftrightarrow (i_3)$ follows from  the inequalities
\begin{equation}\label{int7--}
n^{-r}  \Vert (T_n^*)^{(r)}\Vert_p\lesssim  \w_r(f, 1/n)_p\lesssim \sum_{k=n}^\infty k^{-r-1} \| (T_{k}^*)^{(r)}\|_{p}, \quad
1\le p\le \infty.
\end{equation}
The left-hand side estimate is a corollary of
the well-known Nikolskii-Stechkin inequality $
  \Vert T_n^{(r)}\Vert_p\lesssim n^r \w_r(T_n, 1/n)_p$.
 The right-hand side estimate was proved in  \cite{Zh}.


Jackson and Bernstein approximation theorems
as well as the corresponding
equivalence $(i_1) \Leftrightarrow (i_2)$
are known to be true in various settings.
Surprisingly enough the results involving the smoothness of  approximation processes given in the strong form, i.e., similar to  inequalities \eqref{int7--},
or, even in the weak form, i.e., similar to 
 equivalence $(i_2) \Leftrightarrow (i_3)$,
are much less known in the literature. It is clear that such results provide additional information on smoothness properties of approximants and, therefore, they are useful for applications.  As an example, we mention that the smooth function spaces (Lipschitz, Sobolev, Besov) can be characterize in terms of smoothness of approximation processes.

The main goal of this paper is to present a thoughtful study of interrelations between
 smoothness properties of functions on various domains and smoothness properties of approximation processes.
 In particular, we extend inequalities \eqref{int7--} as follows: for $f \in L_p(\T), 1 < p < \infty,$
\begin{equation*}\label{-optimal-p--}
 \( \sum\limits_{k =n+1}^{\infty} 2^{- k r \tau}
\Vert (T_{2^k}^*)^{(r)}\Vert_p^{\tau}
\)^{\frac{1}{\tau}} \lesssim \omega_{r} \( f,
{2^{-n}} \)_p \lesssim \( \sum\limits_{k =n+1}^{\infty}
2^{- k r \theta} \Vert (T_{2^k}^*)^{(r)}\Vert_p^{\theta}
\)^{\frac{1}{\theta}},
\end{equation*}
where $T_{2^k}^*$ stands for the 
 best approximants, partial sums of the Fourier series, de la Vall\'ee Poussin means, Fej\'er means, etc.

In the general form, our main results state that for
$f \in X$
\begin{equation}\label{-optimal-}
    \(\sum_{k=n+1}^\infty 2^{-k\a\tau}\Vert P_{2^k}(f)\Vert_{Y}^\tau\)^\frac1\tau
    \lesssim \Omega(f,2^{-n\a},X,Y)
    \lesssim
    \(\sum_{k=n+1}^\infty 2^{-k\a\theta}\Vert P_{2^k}(f)\Vert_{Y}^\theta\)^\frac1\theta,
 \end{equation}
 where the parameters $\tau$ and  $\theta$ are
 related to geometry of the space $X$, and, in particular, for $X=L_p$,  $0< p\le\infty$, are given by
  $$\tau=\left\{
                                            \begin{array}{ll}
                                              \max(p,2), & \hbox{$1<p<\infty$,} \\
                                              \infty, & \hbox{otherwise}
                                            \end{array}
                                          \right.,
\qquad
\theta=\left\{
                                            \begin{array}{ll}
                                              \min(p,2), & \hbox{$p<\infty$,} \\
                                              1, & \hbox{$p=\infty$}
                                            \end{array}
                                          \right..
 $$
Here $Y$ is a smooth function space (Sobolev or Besov spaces),
$P_n(f)$ is a suitable (linear or non-linear) approximation method, and $\Omega(f,2^{-n\a},L_p,Y)$ is some measure of smoothness related to the spaces $L_p$ and $Y$.
It is worth mentioning that the classical modulus of smoothness is equivalent to the $K$-functional
  for a couple  $(L_p,W_p^r)$, namely, $K(f,t; L_p(\T), W_p^r(\T))_p\asymp \w_r(f,t)$, see, e.g.,~\cite[p.~177]{DeLo}.
Therefore, as a measure of smoothness 
 it is natural to consider
 the $K$-functional $K(f,2^{-n\a},L_p,Y)$
  in the case $1\le p\le \infty$ and either
   an appropriate modulus of smoothness or a realization of the $K$-functional for any $0< p\le \infty$.




The rest of the  paper is organized as follows.
In Section~\ref{sec3}, we consider general (Banach) spaces and investigate smoothness properties of the best approximants. Using  geometric properties of $X$, 
we obtain sharp inequalities  (\ref{-optimal-}) for appropriate $\t$ and~$\tau$.  
In more detail, if the space $X$ is $\t$-uniformly smooth and $\tau$-uniformly convex, then  (\ref{-optimal-}) holds.


Section~\ref{sec4} studies the smoothness properties of Fourier means of functions from $L_{p,w}(\mathcal{D})$. Our approach is based on Littlewood-Paley-type inequalities and H\"ormander's type multiplier theorems. In particular, inequalities~\eqref{-optimal-} are obtained for a wide class of Fourier multiplier operators, which includes partial sums of Fourier series, de la Val\'ee Poussin means, Fej\'er means, Riesz means, etc. Sharpness of the parameters in~\eqref{-optimal-} will be discussed in Section~\ref{sec9}.

   In Section~\ref{sec2}, we deal with general approximation processes $\{P_{2^n}(f)\}$ and abstract measures of smoothness $\Omega(f,t)_X$ in the metric space $X$. In particular, we treat the case of $X=L_p$ for $0<p<1$.
We prove that
$$
    \left\|\left\{\Omega(P_{2^k}(f),2^{-k})_X
    \right\}_{k\ge n}\right\|_{\ell_\infty}
    \lesssim \Omega(f,2^{-n})_X\lesssim
    \left\|\left\{\Omega(P_{2^k}(f),2^{-k})_X
    \right\}_{k\ge n}\right\|_{\ell_{\l}},
$$
where $\l$ is a parameter related to the geometry of $X$. Let us emphasize that this result holds under very mild conditions on the approximants $P_{2^n}(f)$. Moreover, these inequalities easily imply  the results similar to those given in the equivalence $(i_2) \Leftrightarrow (i_3)$.

In Sections~\ref{sec5}--\ref{sec8}, we illustrate our main results obtained in Sections~\ref{sec2}--\ref{sec4} by several important examples. 
  In particular, in Section~\ref{sec5}, we investigate  
   relationship between smoothness of periodic functions on $\T^d$  and smoothness of 
    the best trigonometric approximants, various Fourier means, and smoothness of interpolation operators. Moreover, we consider  approximations
    in Hardy spaces $H_p(D)$, $0<p\le 1$, and smooth  (Lipschitz, Sobolev) spaces. 
      Section~\ref{sec6} is devoted to approximation processes on $\R^d$. In this case, we study smoothness properties of band-limited functions that approximate functions from $L_{p}(\R^d)$.

      In Section~\ref{sec7}, we deal with functions on $L_{p,w}[-1,1]$, where $w$ is the Jacobi weight. 
       In particular, we study smoothness properties of algebraic polynomials and splines of the best approximation and consider some Fourier means related to Fourier--Jacobi series.

       In Section~\ref{sec8}, we show that the results of Sections~\ref{sec2} and~\ref{sec3} can be applied to study smoothness properties of  non-linear approximation processes. As examples, we treat non-linear wavelet approximation and splines with free knots.

Finally, in Section~\ref{sec9},  we study the optimality of inequalities~\eqref{-optimal-},
  showing that the parameters $\tau$ and $\t$ cannot be improved in general.  Moreover, we define function classes such that
   the right-hand side and the left-hand side sums in~\eqref{-optimal-} (with appropriate values of $\tau$ and $\t$) are equivalent to the corresponding modulus of smoothness.

Throughout the paper, we use the notation
$\, F \lesssim G,$
with $F,G\ge 0$, for the
estimate
$\, F \le C\, G,$ where $\, C$ is a positive constant independent of
the essential variables in $\, F$ and $\, G$ (usually, $f$, $\d$, and $n$). 
 If $\, F \lesssim G$
and $\, G \lesssim F$ simultaneously, we write $\, F \asymp G$ and say that $\, F$
is equivalent to $\, G$.

\bigskip

\bigskip
{\bf{Acknowledgements.}}
The first author was partially supported by DFG project KO 5804/1-1.
The second author was partially supported by MTM 2017-87409-P, 2017 SGR 358, and by the CERCA Programme of the Generalitat de Catalunya.
The authors would like to thank the Isaac Newton Institute for Mathematical Sciences, Cambridge, for support and hospitality during the programme "Approximation, sampling and compression in data science" where part of the work on this paper was undertaken. This work was supported by EPSRC grant no EP/K032208/1.


\bigskip

\bigskip
\newpage

\section{$K$-functionals and smoothness of best approximants}\label{sec3}


Let $(X,Y)$ be a couple of normed function spaces with (semi-)norms $\Vert\cdot\Vert_X$ and $\Vert \cdot\Vert_Y$ respectively and $Y\subset X$. The Peetre $K$-functional for this couple is given by
\begin{equation}\label{K1.1}
  K(f,t;X,Y)=\inf\{\Vert f-g\Vert_X+t\Vert g\Vert_Y\,:\, g\in Y\}
\end{equation}
for any $f\in X$ and $t>0$.

Let $\{G_n\}_{n=1}^\infty$ be a family of subsets of $Y$ such that:

\medskip

\noindent i) $0\in G_1$,

\noindent ii) $G_n\subset G_{n+1}$,

\noindent iii) $G_n=-G_n$,

\noindent iv) the closure of $\{G_n\}_{n=1}^\infty$ in $X$ is $X$.

\medskip

The best approximation of $f\in X$ by elements from $G_n$ is given by
$$
E_n(f)_X=\inf \{\Vert f-g\Vert_X\,:\, g\in G_n\}.
$$

Moreover, we suppose that the family $\{G_n\}$ is such that Jackson and Bernstein type inequalities are valid.  Namely, there are positive constants $c_1$, $c_2$, and $\alpha$ such that for any $n\in \N$ we have
\begin{equation}\label{K1.2}
  E_n(f)_X\le c_1 n^{-\a}\Vert f\Vert_Y,\quad f\in Y,
\end{equation}
\begin{equation}\label{K1.3}
  \Vert g_1-g_2\Vert_Y\le c_2 n^\a\Vert g_1-g_2\Vert_X,\quad g_1,g_2\in G_n.
\end{equation}

The latter condition implies that, for every $g\in G_n$,
\begin{equation}\label{Kd}
  \Vert g\Vert_Y\le c_2 n^\a\Vert g\Vert_X,\quad g\in G_n.
\end{equation}
Clearly, if $G_n$ is a linear space, then (\ref{K1.3}) and (\ref{Kd}) are equivalent.

It is also plain to see that the Jackson-type inequality (\ref{K1.2}) implies the direct approximation theorem given by
\begin{equation}\label{eqJK}
  E_n(f)_X\lesssim K(f,n^{-\a};X,Y),   \quad f\in X,\quad n\in \N.
\end{equation}

Our main goal in this section is to obtain inequalities for $K(f,t;X,Y)$ in terms of the best approximation of $f$ by elements from $G_n$.

In what follows, we denote by $P_n(f)$ an element of the best approximation of $f\in X$ by functions from $G_n$ (assuming it exists), i.e.,
$$
E_n(f)_X=\Vert f-P_n(f)\Vert_X\le \Vert f-g\Vert_X\quad \text{for any} \quad g\in G_n.
$$
An element of the near best approximation of $f\in X$ by functions from $G_n$ is denoted by $Q_n(f)$, i.e., there exists a constant $c>0$  independent of $f$ and $n$ such that
$$
\Vert f-Q_n(f)\Vert_X\le cE_n(f)_X.
$$

One of our main tools is the realization of $K$-functional given by
\begin{equation}\label{R1.1}
  R(f,n^{-\a};X,G_n)=\inf\{\Vert f-g\Vert_X+n^{-\alpha}\Vert g\Vert_Y\,:\, g\in G_n\}.
\end{equation}
Clearly, 
$$
K(f,n^{-\a};X,Y) \le R(f,n^{-\a};X,G_n),\quad f\in X,\quad n\in \N,
$$
but for applications it is important to know when
$$
K(f,n^{-\a};X,Y) \asymp R(f,n^{-\a};X,G_n).
$$
The next proposition describes such cases.

%
%
%
%
%

\begin{proposition}\label{l-real}
Let inequalities~\eqref{K1.2} and \eqref{K1.3} hold. 
Then the following conditions are equivalent:

\begin{enumerate}
  \item[(i)] for every $f\in X$ and $n\in \N$,
\begin{equation}\label{r1--}
R(f,n^{-\a};X,G_n)\lesssim K(f,n^{-\a};X,Y),
\end{equation}

  \item[(ii)] for every $f\in X$ and $n\in \N$,
\begin{equation*}\label{r1--+}
\Vert f-Q_n(f)\Vert_X+n^{-\a}\Vert Q_n(f)\Vert_Y\lesssim K(f,n^{-\a};X,Y).
\end{equation*}

\end{enumerate}
\end{proposition}

Even though Proposition~\ref{l-real} in this form was not mentioned in~\cite{HI}, its proof easily follows from~\cite[Theorem 2.2]{HI} taking into account that by~\eqref{eqJK}, for the near best approximation $Q_n(f)$, we have
$$
\Vert f-Q_n(f)\Vert_X\lesssim E_n(f)_X\lesssim K(f,n^{-\a};X,Y)
$$
for any $f\in X$ and $n\in \N$.

\begin{remark} \emph{It follows from~\cite[Theorem 2.2]{HI} that under conditions of Proposition~\ref{l-real}, assertions (i) and (ii) are equivalent to the following conditions:}
\begin{enumerate}
\item[(iii)] \emph{for every $f\in X$  and $n\in \N$,}
\begin{equation*}\label{r1++}
\Vert P_n(f)\Vert_Y\lesssim n^\a K(f,n^{-\a};X,Y),
\end{equation*}
\item[(iv)] \emph{for every $g\in G_n$  and $n\in \N$,}
\begin{equation*}\label{r1+}
\Vert g\Vert_Y\lesssim n^\a K(g,n^{-\a};X,Y).
\end{equation*}

\end{enumerate}

\end{remark}



The next lemma is a crucial result of this section.

\begin{lemma}\label{thAB}
  Let $f\in X$ and inequalities~\eqref{K1.2}, \eqref{K1.3}, and~\eqref{r1--} hold.

{\rm (A)} Suppose that there exist positive constants $A$ and $\tau$ such that 
\begin{equation}\label{A}
  \Vert f-P_n(f)\Vert_X^\tau\le \Vert f-g\Vert_X^\tau-A \Vert g-P_n(f)\Vert_X^\tau,
\end{equation}
for any $g\in G_n$.
  Then, for any $n\in \N$, we have
  \begin{equation}\label{eq1LP}
    \(\sum_{k=n+1}^\infty 2^{-k\a\tau}\Vert P_{2^k}(f)\Vert_{Y}^\tau\)^\frac1\tau\lesssim K(f,2^{-n\a};X,Y).
  \end{equation}

{\rm (B)} Suppose that there exist  positive constants $B$ and $\t$ such that
\begin{equation}\label{B}
  \Vert f-g\Vert_X^\t\le \Vert f-P_n(f)\Vert_X^\t+B \Vert g-P_n(f)\Vert_X^\t
\end{equation}
for all $g\in G_n$.
  Then, for any $n\in \N$, we have
  \begin{equation}\label{eq1+}
 K(f,2^{-n\a};X,Y)\lesssim \(\sum_{k=n+1}^\infty 2^{-k\a\t}\Vert P_{2^k}(f)\Vert_{Y}^\t\)^\frac1\t.
  \end{equation}
\end{lemma}

\begin{proof} (A)
  Using the representation
  $$
   P_{2^{k}}(f)=\sum_{l=n+1}^k \(P_{2^l}(f)-P_{2^{l-1}}(f)\)+P_{2^n}(f),
  $$
we derive
\begin{equation}\label{eq2LP-}
  \begin{split}
     &\sum_{k=n+1}^\infty 2^{-k\a\tau}\Vert P_{2^k}(f)\Vert_{Y}^\tau\\
     &\lesssim \sum_{k=n+1}^\infty 2^{-k\a\tau}\Vert \sum_{l=n+1}^k P_{2^l}(f)-P_{2^{l-1}}(f)\Vert_{Y}^\tau+2^{-n\a\tau}\Vert P_{2^n}(f)\Vert_{Y}^\tau\\
     &\lesssim \sum_{k=n+1}^\infty 2^{-k\a\tau}\( \sum_{l=n+1}^k \Vert P_{2^l}(f)-P_{2^{l-1}}(f)\Vert_{Y}\)^\tau+2^{-n\a\tau}\Vert P_{2^n}(f)\Vert_{Y}^\tau\\
     &=:L+2^{-n\a\tau}\Vert P_{2^n}(f)\Vert_{Y}^\tau.
  \end{split}
\end{equation}
Next, by Hardy's inequality
\begin{equation}\label{eqHardy}
 \sum_{k=n}^\infty 2^{-k\a} \(\sum_{s=n}^k A_s\)^q \asymp \sum_{k=n}^\infty 2^{-\a k}A_k^q,\quad A_k\ge 0,\quad q>0,
\end{equation}
and  Bernstein's inequality~\eqref{K1.3}, we obtain
\begin{equation}\label{eq2LP}
  \begin{split}
    L\lesssim \sum_{k=n+1}^\infty 2^{-k\a\tau} \Vert P_{2^k}(f)-P_{2^{k-1}}(f)\Vert_{Y}^\tau \lesssim \sum_{k=n+1}^\infty \Vert P_{2^{k}}(f)-P_{2^{k-1}}(f)\Vert_{X}^\tau.
  \end{split}
\end{equation}

Using~\eqref{A} with  $g=P_{2^{k-1}}(f)$ and $n=2^k$, we derive
\begin{equation}\label{eq3LP}
  \Vert P_{2^{k}}(f)-P_{2^{k-1}}(f)\Vert_{X}^\tau\le \frac1{A}\(\Vert f-P_{2^{k-1}}(f)\Vert_X^\tau-\Vert f-P_{2^k}(f)\Vert_X^\tau\).
\end{equation}
Thus, combining~\eqref{eq2LP} and~\eqref{eq3LP} and taking into account that $E_{2^k}(f)_X=\Vert f-P_{2^k}(f)\Vert_X\to 0$ as $k\to\infty$, we have
\begin{equation}\label{eq4LP}
  L\lesssim \Vert f-P_{2^n}(f)\Vert_X^\tau.
\end{equation}

Finally, combining~\eqref{eq2LP-} and~\eqref{eq4LP} and using~Proposition~\ref{l-real}, we obtain \eqref{eq1LP}.

\medskip

(B) By the definition of the $K$-functional, we have
\begin{equation*}\label{eq5-}
  K(f,2^{-n\a};X,Y)\le \Vert f-P_{2^{n+1}}(f)\Vert_X+2^{-n\a}\Vert P_{2^{n+1}}(f)\Vert_{Y}.
\end{equation*}

Thus, to prove~\eqref{eq1+} it is enough to  show that
\begin{equation}\label{eq5LP}
  \Vert f-P_{2^{n+1}}(f)\Vert_X^\t\lesssim \sum_{k=n+1}^\infty 2^{-k\a\t}\Vert P_{2^k}(f)\Vert_{Y}^\t.
\end{equation}
Since $E_{2^k}(f)_X\to 0$ as $k\to \infty$,  we derive
\begin{equation}\label{eq6LP}
  \Vert f-P_{2^{n+1}}(f)\Vert_X^\t=\sum_{k=n+2}^\infty \(\Vert f-P_{2^{k-1}}(f)\Vert_X^\t-\Vert f-P_{2^k}(f)\Vert_X^\t\).
\end{equation}
Next, by the definition of the best approximation,
$$
\Vert f-P_{2^{k-1}}(f)\Vert_X\le \Vert f-P_{2^{k-1}}(P_{2^k}(f))\Vert_X.
$$
Then, inequality~\eqref{B} with $n=2^k$ and $g=P_{2^{k-1}}(P_{2^k}(f))$ and the Jackson inequality~\eqref{K1.2} imply
\begin{equation}\label{eq7+}
  \begin{split}
     \Vert f-P_{2^{k-1}}(f)\Vert_X^\t-&\Vert f-P_{2^k}(f)\Vert_X^\t\\
     &\le \Vert f-P_{2^{k-1}}(P_{2^k}(f))\Vert_X^\t-\Vert f-P_{2^k}(f)\Vert_X^\t\\
     &\le B \Vert P_{2^k}(f)-P_{2^{k-1}}(P_{2^k}(f))\Vert_X^\t\\
     &\lesssim 2^{-(k-1)\a\t}\Vert P_{2^k}(f)\Vert_{Y}^\t.
   \end{split}
\end{equation}
Thus, \eqref{eq6LP} and~\eqref{eq7+} yield~\eqref{eq5LP}, completing  the proof.
\end{proof}

\begin{remark}
{\rm (i)} \emph{It follows from the proof of Lemma~\ref{thAB} that conditions~\eqref{A} and~\eqref{B} can be replaced by the following weaker conditions
\begin{equation*}\label{A++}
  \Vert f-P_{2n}(f)\Vert_X^\tau\le \Vert f-P_{n}(f)\Vert_X^\tau-A \Vert P_{n}(f)-P_{2n}(f)\Vert_X^\tau
\end{equation*}
and
\begin{equation*}\label{B+}
  \Vert f-P_{n}(P_{2n}(f))\Vert_X^\t\le \Vert f-P_{2n}(f)\Vert_X^\t+B \Vert P_{n}(P_{2n}(f))-P_{2n}(f)\Vert_X^\t,
\end{equation*}
respectively.}

{\rm (ii)} \emph{Note that by triangle inequality, estimate~\eqref{B} is always valid with $\t=B=1$.}

{\rm (iii)} \emph{Lemma~\ref{thAB} remains valid without assumption~\eqref{r1--} with the realization $R(f,2^{-n\a};X,Y)$
in place of  the $K$-functional $K(f,2^{-n\a};X,Y)$ in~\eqref{eq1LP} and~\eqref{eq1+}.}
%
%
%
%
\end{remark}


In what follows, we need some terminology from the theory of Banach spaces (see, e.g.,~\cite[Ch.~IV]{DGZ}).
Let $X$ be a Banach space with the norm $\Vert \cdot \Vert=\Vert \cdot \Vert_X$. The moduli of convexity and smoothness of $X$ are defined respectively by
$$
\d_X(\e)=\inf\left\{1-\bigg\Vert \frac{x+y}{2} \bigg\Vert\,:\, \Vert x\Vert=\Vert y\Vert=1\quad \text{and}\quad \Vert x-y\Vert=\e\right\},\quad 0\le \e\le 2,
$$
and
$$
\rho_X(t)=\sup\left\{\frac12(\Vert x+y\Vert+\Vert x-y\Vert)-1\,:\, \Vert x\Vert=1,\quad \Vert y\Vert=t\right\},\quad t>0.
$$
Let $\tau,\t>1$ be real numbers. Then $X$ is said to be $\tau$-uniformly convex (respectively, $\t$-uniformly smooth) if there exists a constant $c>0$ such that $\d_X(\e)\ge c\e^\tau$ (respectively, $\rho_X(t)\le c t^\t$).
Note that by the Day--Nordlander theorem we always have  $\t\le 2\le \tau$, see, e.g.,~\cite{I} or \cite{N}.


\begin{theorem}\label{mainmain}
 Let  $G_n$ be convex, $f\in X$, and inequalities~\eqref{K1.2}, \eqref{K1.3}, and~\eqref{r1--} hold.

{\rm (A)} Suppose $X$ is $\tau$-uniformly convex for some $\tau>1$. Then, for any $n\in \N$, we have
  \begin{equation}\label{eq1LPmod}
    \(\sum_{k=n+1}^\infty 2^{-k\a\tau}\Vert P_{2^k}(f)\Vert_{Y}^\tau\)^\frac1\tau\lesssim K(f,2^{-n\a};X,Y).
  \end{equation}

{\rm (B)} Suppose $X$ is $\t$-uniformly smooth for some $\t>1$.  Then, for any $n\in \N$, we have
  \begin{equation}\label{eq1+mod}
 K(f,2^{-n\a};X,Y)\lesssim \(\sum_{k=n+1}^\infty 2^{-k\a\t}\Vert P_{2^k}(f)\Vert_{Y}^\t\)^\frac1\t.
  \end{equation}
\end{theorem}

\begin{proof}
(A) Since $X$ is $\tau$-uniformly convex, then  there exists a constant $c>0$ such that for all $x,y\in X$ and $t\in [0,1]$
\begin{equation}\label{3.1Xu}
  \Vert t x+(1-t)y\Vert^\tau\le t \Vert x\Vert^\tau\ +(1-t)\Vert y\Vert^\tau\-W_\tau(t)c\Vert x-y\Vert^\tau,
\end{equation}
where $\Vert \cdot \Vert=\Vert \cdot\Vert_X$ and $W_\tau(t)=t(1-t)^\tau+t^\tau (1-t)$ (see the proof of Theorem~1 in~\cite{Xu}, see also~\cite{PS} and~\cite{Sm}).
Consider the following Gateaux derivative at $y$ in the direction $x-y$
$$
{\rm g}_\tau(y,x-y)=\lim_{t\to +0}\frac{\Vert y-t(x-y) \Vert^\tau-\Vert y\Vert^\tau}{t}.
$$
Dividing both sides of~\eqref{3.1Xu} by $t\in (0,1)$ and taking limit as $t\to +0$, we get
\begin{equation*}
  {\rm g}_\tau(y,x-y) \le \Vert x\Vert^\tau-\Vert y\Vert^\tau-c\Vert x-y\Vert^\tau.
\end{equation*}

Now let $g\in G_n$. Replacing $x$ by $f-g$ and $y$ by $f-P_n(f)$, we have that
$$
{\rm g}_\tau(f-P_n(f),P_n(f)-g)\le \Vert f-g\Vert^\tau-\Vert f-P_n(f)\Vert^\tau-c\Vert P_n(f)-g\Vert^\tau.
$$
By the Kolmogorov criterion, see, e.g.,~\cite[p.~90]{Si}, we have ${\rm g}_\tau(f-P_n(f),P_n(f)-g)=0$, which implies~\eqref{A}. Thus, using Lemma~\ref{thAB}, we get~\eqref{eq1LPmod}.


(B) The proof of~\eqref{eq1+mod} is similar. We only note that by~\cite[Theorem 1$'$]{Xu}, $X$ is $\t$-uniformly smooth if and only if there exists a constant $d>0$ such that
\begin{equation}\label{3.1Xu2}
  \Vert t x+(1-t)y\Vert^\t\ge t \Vert x\Vert^\t\ +(1-t)\Vert y\Vert^\t\-W_\t(t)d\Vert x-y\Vert^\t.
\end{equation}
Then, as above, we derive
$$
{\rm g}_\tau(f-P_n(f),P_n(f)-g)\ge \Vert f-g\Vert^\t-\Vert f-P_n(f)\Vert^\t-c\Vert P_n(f)-g\Vert^\t
$$
and apply the Kolmogorov criterion. Lemma~\ref{thAB} completes the proof.
\end{proof}


Let us give two important examples of Banach space $X$ to illustrate Theorem~\ref{mainmain}, namely,
Lebesgue and Orlicz spaces.

\begin{proposition}\label{remLp} {\rm (See \cite[p.~63]{LT}.)}
Let $X$ be an abstract $L_p$ space with $1<p<\infty$, i.e. let $X$ be a Banach lattice for which
$$
\Vert x+y\Vert^p=\Vert x\Vert^p+\Vert y\Vert^p,
$$
whenever $x,y\in X$ and $\min(x,y)=0$.
Then there exists a constant $c>0$ such that $\d_X(\e)\ge c\e^{\max(2,p)}$ for all $0\le \e\le 2$ and $\rho_X(t)\le c t^{\min(2,p)}$ for all $t>0$.
\end{proposition}

Making use of  Theorem~\ref{mainmain} and Proposition~\ref{remLp}, we obtain the following result.

\begin{theorem}\label{corthAB}
  Let inequalities~\eqref{K1.2}, \eqref{K1.3}, and~\eqref{r1--} be valid for $X=L_p$, $1<p<\infty$, and let $G_n$ be convex.
Then, for any $f\in L_p$ and $n\in \N$, we have
  \begin{equation*}\label{eq1LPc}
    \(\sum_{k=n+1}^\infty 2^{-k\a\tau}\Vert P_{2^k}(f)\Vert_{Y}^\tau\)^\frac1\tau\lesssim K(f,2^{-n\a};L_p,Y),\quad \tau=\max(2,p),
  \end{equation*}
and
  \begin{equation*}\label{eq1+c}
 K(f,2^{-n\a};L_p,Y)\lesssim \(\sum_{k=n+1}^\infty 2^{-k\a\t}\Vert P_{2^k}(f)\Vert_{Y}^\t\)^\frac1\t,\quad  \t=\min(2,p).
  \end{equation*}
\end{theorem}

In Section~\ref{secOpt}, we will see that the parameters $\tau$ and $\t$ in Theorem~\ref{corthAB} are optimal.

\begin{corollary}\label{corRathor}
  Let inequalities~\eqref{K1.2}, \eqref{K1.3}, and~\eqref{r1--} be valid for $X=L_p$, $1<p<\infty$,  and let $G_n$ be convex.
Then, for any $f\in L_p$, the following assertions are equivalent:

{\rm (i)}   for any $n\in \N$
\begin{equation*}
  K(f,2^{-n\a};L_p,Y)\asymp 2^{-n\a\t}\Vert P_{2^n}(f)\Vert_{Y},
\end{equation*}

{\rm (ii)} for any $n\in \N$
\begin{equation*}
  \sum_{k=n}^\infty 2^{-k\a\t} \Vert P_{2^k}(f)\Vert_{Y} \lesssim 2^{-n\a\t} \Vert P_{2^n}(f)\Vert_{Y}.
\end{equation*}

\end{corollary}

\begin{proof}
The proof easily follows from Theorem~\ref{corthAB} and~\eqref{Omegla}.
\end{proof}

Finally, we consider Orlicz spaces. Recall that the Orlicz function $M(t)$ on $[0,\infty)$ is an increasing convex function satisfying $M(0)=0$. We assume that $M$
satisfies $\D_2$-condition, that is, $M(2t)\le c M(t)$ for all $t>0$. The Orlicz class of functions $X=X_M$ on some domain $\mathcal{D}$ with a positive measure $d\mu(x)$ is the class of functions $f$, for which
\begin{equation}\label{spo1}
  \int_{\mathcal{D}} M(|f(x)|)d\mu(x)<\infty,
\end{equation}
and the  (Luxemburg) norm  is
\begin{equation}\label{spo2}
\Vert f\Vert_{X_M}=\inf\left\{\s>0\,:\, \int_{\mathcal{D}} M(|f(x)|/\s)d\mu(x)\le 1\right\}.
\end{equation}

\begin{proposition}\label{prO}
{\rm (A)} Suppose that $M(u)$ is an Orlicz function such that $M(u^{1/\tau})$ is concave for some $\tau$, $2\le \tau<\infty$, and $M(lt)\le \frac12 M(t)$ for some $l<1$. Then there exists an Orlicz function $N(u)$ such that $C^{-1} N(u)\le M(u)\le C N(u)$ and $\delta_{X_N}(\e)\ge c\e^\tau$ with the norm of the space $X_N$ given by
\begin{equation}\label{obia}
 \Vert f\Vert_{X_N}=\inf\left\{\s>0\,:\, \int_{\mathcal{D}} N(|f(x)|/\s)d\mu(x)\le 1\right\}.
\end{equation}
{\rm (B)} Suppose that $M(u)$ is an Orlicz function such that $M(u^{1/\t})$ is convex for some $\t$, $1<\t\le 2$. Then there exists an Orlicz function $N(u)$ such that $C^{-1} N(u)\le M(u)\le C N(u)$ and $\rho_{X_N}(t)\le ct^\t$ with the norm of the space $X_N$ given by~\eqref{obia}.
\end{proposition}

\begin{proof}
The proof of~(B) can be found in~\cite[Lemma~2.2]{DP}. Assertion (A) can be proved similarly employing Theorem~1 from~\cite{MT}.
\end{proof}

Using Theorem~\ref{mainmain} and Proposition~\ref{prO}, we obtain the following result.

\begin{theorem}\label{corthAB2}
  Let inequalities~\eqref{K1.2}, \eqref{K1.3}, and~\eqref{r1--} be valid for the Orlicz space $X=X_M$ defined by~\eqref{spo1} and~\eqref{spo2}, and let $G_n$ be convex.


{\rm (A)} Suppose that the function $M$ and the parameter $\tau$ are the same as in Proposition~\ref{prO} (A).
Then, for any $f\in X$ and $n\in \N$, we have
  \begin{equation*}\label{eq1LPc}
    \(\sum_{k=n+1}^\infty 2^{-k\a\tau}\Vert P_{2^k}(f)\Vert_{Y}^\tau\)^\frac1\tau\lesssim K(f,2^{-n\a};X,Y).
  \end{equation*}

{\rm (B)} Suppose that the function $M$ and the parameter $\t$ are the same as in Proposition~\ref{prO} (B).
Then, for any $f\in X$ and $n\in \N$, we have
  \begin{equation*}\label{eq1+c}
 K(f,2^{-n\a};X,Y)\lesssim \(\sum_{k=n+1}^\infty 2^{-k\a\t}\Vert P_{2^k}(f)\Vert_{Y}^\t\)^\frac1\t.
  \end{equation*}

\end{theorem}

%
%



\section{Smoothness of Fourier multiplier operators}\label{sec4} 

\subsection{Realization and Littlewood-Paley-type inequality}\label{SecReal}

First we introduce basic notations and collect auxiliary results. We follow the discussion in the paper~\cite{DDT}.

We assume that $Q(D)$ is a self-adjoint operator, that is,
$
\la Q(D) f,g\ra = \la f,Q(D)g\ra$ whenever $Q(D)f,Q(D)g\in
L_{2,w}(\mathcal{D}),
$
where
$$
\Vert  f\Vert_{L_{p,w}(\mathcal{D})} = \Vert f\Vert_{p,w}=\(\int_{\mathcal{D}} \vert  f\vert  ^pw \)^\frac1p
$$
and
$\la f,g\ra = \int_{\SD} fgw.$  Let $\lambda_k$ be the eigenvalues of $Q(D),$  satisfying
$$
0\le \lambda_0<\lambda_k<\lambda_{k+1}, \q G_k =\{\varphi
:Q(D)\varphi  = \lambda_k\varphi  \},
$$
$G_k$ is finite dimensional, $G_k\subset L_{p,w}(\SD) $ for $1\le p\le
\infty  $ and ${\rm span}\,\cup_k G_k$ is dense in $L_{p,w}(\SD)$
for $1\le p<\infty  .$  Examples of such operators and matching
spaces are: $-\big(\frac{d}{dx}\big)^2$ for $L_p(\T);$
$-\,\frac{d}{dx}\,(1-x^2)\,\frac{d}{dx}$ for $L_p[-1,1];$
$-\Delta  +\vert  x\vert  ^2$, where $\Delta  $ is the
Laplacian for $L_p(\R^d);$ and $-w^{-1}_{\alpha  ,\beta
}\;\frac{d}{dx}\;w_{\alpha  \beta  }(1-x^2)\,\frac{d}{dx}$ for
$L_{p,w_{\alpha  ,\beta  }}[-1,1]$, where $w_{\alpha  ,\beta  }(x) =
(1-x)^\alpha  (1+x)^\beta  $ with $\alpha  ,\beta  >-1$.

We define
$$
A_k f =\sum^{d_k}_{\ell=1} \, \langle f,\psi_{k,\ell} \rangle\,\psi_{k,\ell}, 
$$
where $d_k$ is the dimension of $G_k$ and $\{\psi  _{k,\ell}\}$ an
orthonormal basis of $G_k$ in $L_{2,w}(\SD).$

For $f\in L_{p,w}(\SD),$ $f\sim \sum_{k=0}^\infty A_k
f,$ we define $Q(D)^\gamma  $ by
$$
Q(D)^\gamma  f \sim \sum_k\lambda_k^\gamma  A_k f
$$
and we say that $Q(D)^\gamma  f\in L_{p,w}(\SD)$ if there exists $g\in L_{p,w}(\SD)$
such that $\lambda_k^\gamma  A_k f = A_k g.$

In what follows, we suppose that $\lambda_k\asymp k^\sigma$ for some positive $\s>0$. Note that in the example above $\sigma=2$ except for the eigenvalues
of $-\Delta  +\vert  x\vert  ^2$ where $\sigma=1$ (see~\cite{Di}).

As usual, we define the $K\text{-functional}\; K_\gamma  \big(f,Q(D),t^{\sigma  \gamma
}\big)_{p,w}$ by
\begin{equation}\label{Kfunc}
  K_\gamma  \big(f,Q(D),t^{\sigma  \gamma  }\big)_{p,w}: =\inf_{Q(D)^\gamma
g\in L_{p,w}(\mathcal{D})}\; \{\Vert  f-g\Vert  _{L_{p,w}(\mathcal{D})} +
t^{\sigma  \gamma  } \Vert  Q(D)^\gamma g \Vert
_{L_{p,w}(\mathcal{D})}\}.
\end{equation}

In this section, we consider approximation processes, which are defined by means of the Fourier multiplier operator $T_{\mu}$ given by
$$
T_{\mu} f \sim \sum^\infty  _{k=0} \mu  _k
A_k f \q\text{for}\q f\sim \sum^\infty  _{k=0} A_k f.
$$
We will use the following assumption related to a H\"ormander-Mikhlin-type theorem.

\begin{assumption}\label{H}
  For some $\ell_0\ge 0$, the condition
\begin{equation}\label{eq(3.7)}
  \vert  \Delta  ^\ell \mu  _k\vert  \le A(k+1)^{-\ell}\q\text{for}\q
0\le\ell \le \ell_0,
\end{equation}
where
$$
\Delta  ^0 \mu  _k = \mu  _k, \q \Delta  \mu  _k = \mu  _{k+1}
-\mu  _k\q\text{and}\q  \Delta  ^\ell \mu  _k = \Delta  (\Delta
^{\ell-1}\mu  _k),
$$
implies
\begin{equation*}\label{eq(3.8)}
  \Vert  T_{\mu} f\Vert  _{L_{p,w}(\SD)} \le
C\big(A,L_{p,w}(\SD),\{G_k\}\big)\Vert  f\Vert  _{L_{p,w}(\SD)}\,, \q
1 < p <\infty.
\end{equation*}
\end{assumption}

It is clear that under Assumption~\ref{H}, the de~la~Vall\'ee Poussin-type operator
$$
\eta  _Nf: =\sum^\infty  _{k=0} \eta  \Big(\frac kN\Big) A_k,\quad f\sim \sum^\infty  _{k=0} A_k f,
$$
 satisfies
\begin{equation}\label{f4}
\Vert  \eta  _N f\Vert_{p,w}\le A\Vert  f\Vert_{p,w}.
\end{equation}
Here and in what follows, we assume that
$$
\eta  (\xi  )\in C^\infty  [0,\infty  ), \q
\eta  (\xi  ) = \left\{
                  \begin{array}{ll}
                    1, & \hbox{$\xi  \le 1/2$,} \\
                    0, & \hbox{$\xi  \ge 1$.}
                  \end{array}
                \right.
$$

Moreover, the following realization result (see~\cite[Theorem 7.1]{Di}) holds:
\begin{equation}\label{f3}
  K_\gamma  \big(f,Q(D),\lambda_N^{-\gamma  }\big)_{p,w}
\asymp\Vert  f-\eta  _Nf\Vert_{p,w} +\lambda_N^{-\gamma  }\Vert
Q(D)^\gamma  \eta  _Nf\Vert_{p,w}. 
\end{equation}

Denote
\begin{equation}\label{LPVP}
  \theta  _0 (f):=\eta  _1f \q \text{\rm and} \q \theta  _j(f): =\eta  _{2^j}f-\eta
_{2^{j-1}} f\q\text{for}\q j>0.
\end{equation}

The following Littlewood--Paley-type theorem plays a crucial role in our further study.

\begin{theorem}\label{J3.0} {\rm (See \cite[Theorem 2.1]{DaDi}, \cite[Theorem 3.1]{DDT}.)}
Let $f\in L_{p,w}(\SD)$, $1<p<\infty$, and Assumption~\ref{H} be satisfied,
then
\begin{equation*}\label{eq3.16}
  \Big\Vert  \Big\{\sum^\infty  _{j=0}
\big(\theta  _j(f)\big)^2\Big\}^{1/2}\Big\Vert  _{L_{p,w}(\SD)} \asymp\Vert   f\Vert  _{L_{p,w}(\SD)}.
\end{equation*}
If, in addition, $\gamma>0$, then
\begin{equation}\label{eq3.16'}
\Big\Vert  \Big\{\sum^\infty  _{j=1} \big(2^{j\gamma  \sigma  }
\theta  _j(f)\big)^2\Big\}^{1/2}\Big\Vert  _{L_{p,w}(\SD)} \asymp\Vert  Q(D)^\gamma  f\Vert  _{L_{p,w}(\SD)}.
\end{equation}
\end{theorem}

%

\subsection{Smoothness of the de~la~Vall\'ee Poussin means 
 in $L_{p,w}$}

\begin{theorem}\label{thLP}
  Let $f\in L_{p,w}(\mathcal{D})$, $1<p<\infty$, $\gamma>0$, $\tau=\max(2,p)$, $\theta=\min(2,p)$, $n\in \N$, and Assumption~\ref{H} hold.
Then
  \begin{equation}\label{L}
    \(\sum_{k=n+1}^\infty 2^{-\s \g \tau k} \Vert Q(D)^\g \eta_{2^k} f\Vert_{p,w}^\tau\)^\frac1\tau \lesssim K_\g(f,Q(D),2^{-n\g \s})_{p,w}
  \end{equation}
and
    \begin{equation}\label{R}
   K_\g(f,Q(D),2^{-n\g \s})_{p,w}\lesssim \(\sum_{k=n+1}^\infty 2^{-\s \g \t k} \Vert Q(D)^\g \eta_{2^k} f\Vert_{p,w}^\t\)^\frac1\t.
  \end{equation}
\end{theorem}

\begin{proof}
  Denote $\a=\s \g$ and
  $$
  I^\tau=\sum_{k=n+1}^\infty 2^{-\a \tau k} \Vert Q(D)^\g \eta_{2^k} f\Vert_{p,w}^\tau.
  $$
Then
\begin{equation}\label{f5}
  \begin{split}
     I^\tau&\lesssim\sum_{k=n+1}^\infty 2^{-\a \tau k} \Vert Q(D)^\g (\eta_{2^k} f-\eta_{2^n} f)\Vert_{p,w}^\tau+2^{-n \a\tau} \Vert Q(D)^\g \eta_{2^n} f\Vert_{p,w}^\tau\\
     &=J+2^{-n \a\tau} \Vert Q(D)^\g \eta_{2^n} f\Vert_{p,w}^\tau.
   \end{split}
\end{equation}
By \eqref{eq3.16'}, we have
\begin{equation*}
  \begin{split}
     J&\lesssim \sum_{k=n+1}^\infty 2^{-k\a \tau}
     \bigg\{
     \int_{\mathcal{D}}
     \bigg(
     \sum_{j=1}^\infty 2^{2\a j} \(\t_j(\eta_{2^k}f-\eta_{2^n}f)\)^2
     \bigg)^\frac p2 w
     \bigg\}^\frac\tau p
     \\
     &=\sum_{k=n+1}^\infty 2^{-k\a \tau} \bigg\{\int_{\mathcal{D}}
     \bigg(
     \sum_{j=n}^{k+1} 2^{2\a j} \(\t_j(\eta_{2^k}f-\eta_{2^n}f)\)^2
     \bigg)^\frac p2 w
     \bigg\}^\frac\tau p\\
     &\lesssim \sum_{k=n+1}^\infty 2^{-k\a \tau} \Big\{2^{\a n} \Vert \t_n(\eta_{2^k}f-\eta_{2^n}f)\Vert_{p,w}+2^{\a (n+1)} \Vert \t_{n+1}(\eta_{2^k}f-\eta_{2^n}f)\Vert_{p,w} \Big\}^\tau\\
     &+\sum_{k=n+1}^\infty 2^{-k\a \tau} \bigg\{\int_\mathcal{D} \Big(\sum_{j=n+2}^{k-1} 2^{2j\a} \t_j(f)^2\Big)^\frac p2 w \bigg\}^\frac\tau p\\
     &+\sum_{k=n+1}^\infty 2^{-k\a \tau}
     \Big\{2^{k\a}\Vert \t_k(\eta_{2^k}f-\eta_{2^n}f)\Vert_{p,w}+2^{(k+1)\a}\Vert \t_{k+1}(\eta_{2^k}f-\eta_{2^n}f)\Vert_{p,w}\Big\}^\tau\\
     &=J_1+J_2+J_3.
   \end{split}
\end{equation*}

Let us estimate the first sum~$J_1$. By~\eqref{f4}, we have
\begin{equation*}
  \begin{split}
      \Vert \t_{j}(\eta_{2^k}f-\eta_{2^n}f)\Vert_{p,w}&\le 2A \Vert \eta_{2^k}f-\eta_{2^n}f\Vert_{p,w}\\
      &\le 2A (\Vert f-\eta_{2^n}f\Vert_{p,w}+\Vert f-\eta_{2^k}f\Vert_{p,w}).
   \end{split}
\end{equation*}
In light of  the fact that
$$
\eta_{2^k} (\eta_{2^n} f)=\eta_{2^n}f\quad\text{for}\quad  k\ge n+1,
$$
we derive
\begin{equation}\label{ZZZ}
  \begin{split}
\Vert f-\eta_{2^k}f\Vert_{p,w}&=\Vert f-\eta_{2^n}f+\eta_{2^k}(\eta_{2^n}f-f)\Vert_{p,w}\\
&\le (1+A)\Vert f-\eta_{2^n}f\Vert_{p,w}.
   \end{split}
\end{equation}
Therefore,
$$
\Vert \t_j(\eta_kf-\eta_nf)\Vert_{p,w}\lesssim \Vert f-\eta_{2^n}f\Vert_{p,w}
$$
and we get
$$
J_1\lesssim \Vert f-\eta_{2^n}f\Vert_{p,w}.
$$
Regarding $J_2$, we note that
$$
\t_j(f)=\t_j(f-\eta_{2^n}f)\quad\text{for}\quad j\ge n+2,
$$
and, therefore,
$$
J_2=\sum_{k=n+1}^\infty 2^{-k\a \tau} \bigg\{\int_\mathcal{D} \Big(\sum_{j=n+2}^{k-1} 2^{2j\a} \t_j(f-\eta_{2^n}f)^2\Big)^\frac p2 w \bigg\}^\frac\tau p.
$$
Dealing with $J_3$, we observe that $\t_k(\eta_{2^n}f)=\eta_{2^n}(\t_k (f))$. Then
\begin{equation*}
  \begin{split}
      \Vert \t_k(\eta_{2^k}f-\eta_{2^n}f)\Vert_{p,w}&\le \Vert \t_k(\eta_{2^k}f-f)\Vert_{p,w}+\Vert \t_k(\eta_{2^n}f-f)\Vert_{p,w}\\
      &=\Vert \eta_{2^k}(\t_k (f))-\t_k(f)\Vert_{p,w}+\Vert \t_k(\eta_{2^n}f-f)\Vert_{p,w}\\
      &\lesssim \Vert \t_k(\eta_{2^n}f-f)\Vert_{p,w},
   \end{split}
\end{equation*}
where in the last estimate we used~\eqref{ZZZ} with $\t_k(f)$ in place of $f$.

Combining the above inequalities, we obtain that
\begin{equation*}\label{f6-}
  J\lesssim \sum_{k=n+1}^\infty 2^{-k\a\tau}\bigg\{\int_\mathcal{D} \bigg(\sum_{j=n}^{k+1} 2^{2j\a}\(\t_j(f-\eta_{2^n}f)\)^2 \bigg)^\frac p2 w\bigg\}^\frac\tau p
             +\Vert f-\eta_{2^n}f\Vert_{p,w}^\tau.
\end{equation*}
Next, using Minkowski's inequality with $\frac\tau p\ge 1$, Hardy's inequality~\eqref{eqHardy},
the inequality $\Vert \{a_k\}\Vert_{\ell_\tau}\le \Vert \{a_k\}\Vert_{\ell_2}$, and Theorem~\ref{J3.0}, we get
\begin{equation*}
  \begin{split}
     \sum_{k=n+1}^\infty 2^{-k\a\tau}&\bigg\{\int_\mathcal{D} \bigg(\sum_{j=n}^{k+1} 2^{2j\a}\(\t_j(f-\eta_{2^n}f)\)^2 \bigg)^\frac p2 w\bigg\}^\frac\tau p\\
     &\lesssim \bigg\{\int_\mathcal{D} \Big[\sum_{k=n+1}^\infty 2^{-k\a\tau} \Big(\sum_{j=n}^{k+1}2^{2j\a}(\t_j(f-\eta_{2^n}f))^2 \Big)^\frac\tau 2 \Big]^\frac p\tau w \bigg\}^\frac\tau p\\
     &\lesssim \bigg\{\int_\mathcal{D} \Big[\sum_{j=n}^\infty |\t_j(f-\eta_{2^n}f)|^\tau \Big]^\frac p\tau w \bigg\}^\frac \tau p\\
     &\lesssim \bigg\{\int_\mathcal{D} \Big[\sum_{j=n}^\infty |\t_j(f-\eta_{2^n}f)|^2 \Big]^\frac p2 w \bigg\}^\frac \tau p\\
     &\lesssim \Vert f-\eta_{2^n}f\Vert_{p,w}^\tau.
  \end{split}
\end{equation*}
Therefore,
\begin{equation}\label{f6}
  J\lesssim\Vert f-\eta_{2^n}f\Vert_{p,w}^\tau.
\end{equation}

In light of~\eqref{f3}, estimates \eqref{f5} and~\eqref{f6} imply
\begin{equation*}
\begin{split}
    I^\tau&\lesssim \Vert f-\eta_{2^n}f\Vert_{p,w}^\tau+2^{-n \a\tau} \Vert Q(D)^\g \eta_{2^n} f\Vert_{p,w}^\tau\\
    &\lesssim K_\g(f,Q(D),2^{-n\g \s})_{p,w}^\tau,
\end{split}
\end{equation*}
which proves~\eqref{L}.

\medskip

Let us prove~\eqref{R}. By~\eqref{f3}, we have
\begin{equation}\label{f8}
  K_\g(f,Q(D),2^{-n\g \s})_{p,w}^\t\lesssim \Vert f-\eta_{2^n}f\Vert_{p,w}^\t+2^{-n \a\t} \Vert Q(D)^\g \eta_{2^n} f\Vert_{p,w}^\t.
\end{equation}
By Theorem~\ref{J3.0}, taking into account that
$$
(\t_j(f-\eta_{2^n}f))^2\le 4 (\t_j(f))^2+4(\t_j(\eta_{2^n}f))^2,
$$
$$
\t_j(\eta_{2^n}f)=0\quad\text{for}\quad j\ge n+2,
$$
$$
\Vert \t_j(\eta_{2^n}f)\Vert_{p,w}\le A\Vert \t_j(f)\Vert_{p,w},
$$
and
$\Vert \{a_k\}\Vert_{\ell_2}\le \Vert \{a_k\}\Vert_{\ell_\t}$,
we derive
\begin{equation*}
  \begin{split}
      \Vert f-\eta_{2^n}f\Vert_{p,w}^\t
      &\lesssim \(\int_{\mathcal{D}} \bigg[\sum_{j=n}^\infty \t_j(f-\eta_{2^n}f)^2\bigg]^\frac p2 w\)^\frac \t p\\
      &\lesssim \(\int_{\mathcal{D}} \bigg[\sum_{j=n}^\infty \t_j(f)^2\bigg]^\frac p2 w\)^\frac \t p\lesssim \(\int_{\mathcal{D}} \bigg[\sum_{j=n}^\infty |\t_j(f)|^\t\bigg]^\frac p\t w\)^\frac \t p\\
      &= \(\int_{\mathcal{D}} \bigg[\sum_{j=n}^\infty  2^{-j\a\t}  \(\t_j(f)^2 2^{2j\a}\)^\frac\t2\bigg]^\frac p2 w\)^\frac \t p\\
      &\lesssim \Bigg(\int_\mathcal{D} \bigg[ \sum_{j=n}^\infty 2^{-j\a\t}\bigg(  \sum_{k=n}^j \t_k(f)^2 2^{2k\a}\\
      &\qquad\quad+\t_{j+1}(\eta_{2^{j+1}}f)^2 2^{2(j+1)\a}+\t_{j+2}(\eta_{2^{j+1}}f)^2 2^{2(j+2)\a}    \bigg)^\frac \t2       \bigg]^\frac p\t w\Bigg)^\frac \t p\\
      &=\Bigg(\int_\mathcal{D} \bigg[ \sum_{j=n}^\infty 2^{-j\a\t}\bigg(  \sum_{k=n}^{j+2} \t_k(\eta_{2^{j+1}}f)^2 2^{2k\a}
         \bigg)^\frac \t2       \bigg]^\frac p\t w\Bigg)^\frac \t p.
   \end{split}
\end{equation*}
Next, Minkowski's inequality with $\frac p\t\ge 1$ and Theorem~\ref{J3.0} (see~\eqref{eq3.16'}), yield
\begin{equation}\label{f9}
  \begin{split}
     \Vert f-\eta_{2^n}f\Vert_{p,w}^\t &\lesssim \sum_{j=n}^\infty 2^{-j\a\t} \Bigg(\int_\mathcal{D} \bigg[ \sum_{k=n}^{j+2} \t_k(\eta_{2^{j+1}}f)^2 2^{2k\a}
                \bigg]^\frac p2 w\Bigg)^\frac \t p\\
     &\lesssim \sum_{j=n}^\infty 2^{-j \a\t} \Vert Q(D)^\g \eta_{2^{j+1}} f\Vert_{p,w}^\t \lesssim \sum_{j=n}^\infty 2^{-j \a\t} \Vert Q(D)^\g \eta_{2^{j}} f\Vert_{p,w}^\t.
  \end{split}
\end{equation}


Finally, combining \eqref{f8} and \eqref{f9}, we derive~\eqref{R}.
\end{proof}

\begin{corollary}\label{CCCC}
Under the conditions of Theorem~\ref{thLP}, we have
  \begin{equation*}
    \(\sum_{k=n+1}^\infty  k^{-\s \g \tau -1} \Vert Q(D)^\g \eta_{k} f\Vert_{p,w}^\tau\)^\frac1\tau \lesssim K_\g(f,Q(D),n^{-\g \s})_{p,w}
  \end{equation*}
and
    \begin{equation*}
   K_\g(f,Q(D),n^{-\g \s})_{p,w}\lesssim \(\sum_{k=n+1}^\infty k^{-\s \g \t-1 } \Vert Q(D)^\g \eta_{k} f\Vert_{p,w}^\t\)^\frac1\t.
  \end{equation*}
\end{corollary}

\begin{proof}
The proof easily follows from inequalities~\eqref{L} and~\eqref{R} and the fact that
$$\Vert Q(D)^\g \eta_{\mu} f\Vert_{p,w}\lesssim \Vert Q(D)^\g \eta_{\nu} f\Vert_{p,w},\qquad \nu\ge 2\mu.$$
The latter holds in light of boundedness of
the de~la~Vall\'ee Poussin-type operator in $L_{p,w}$
given by~\eqref{f4} and the fact that $\eta_{\mu}(\eta_\nu f)=\eta_\mu f$ for $\nu\ge 2\mu$. We also take into account that
$K_\g(f,Q(D),2t)_{p,w}\asymp K_\g(f,Q(D),t)_{p,w}$ for any
 $t>0$.
\end{proof}

\subsection{General Fourier multiplier operators}

In this subsection, we extend Theorem~\ref{thLP} considering general Fourier multiplier operators given by
$$
\Psi_n f\sim \sum_{k=0}^\infty \psi\(\frac kn\) A_k f,
$$
where a function $\psi\,: [0,\infty) \to \R$ is such that $\supp \psi \subset [0,1)$.
Together with the operator $\Psi_n$, additionally  assuming that $\psi(x)\neq 0$ for all $x\in [0,2^{-m}]$ for some $m\in \Z_+$, we will also use the operator
$$
\widetilde{\Psi}_n\sim \sum_{k=0}^\infty {\widetilde{\psi}\(\frac kn\)}A_k f,\quad \quad \widetilde{\psi}(\xi)=\frac{\eta(\xi)}{\psi(2^{-m}\xi)},
$$
which plays a role of the inverse operator to $\Psi_n$.

\begin{theorem}\label{thPsi}
  Suppose that the conditions of Theorem~\ref{thLP} are satisfied.

{\rm (A)} Let the operators $\Psi_{2^n}$  be such that, for any $f\in L_{p,w}(\mathcal{D})$ and $n\in\N$,
\begin{equation}\label{dopa}
  \Vert \Psi_{2^n} f\Vert_{p,w}\le C \Vert f\Vert_{p,w},
\end{equation}
where the constant $C$ does not depend on $f$ and $n$. Then
  \begin{equation}\label{L+}
    \(\sum_{k=n+1}^\infty 2^{-\s \g \tau k} \Vert Q(D)^\g \Psi_{2^k} f\Vert_{p,w}^\tau\)^\frac1\tau \lesssim K_\g(f,Q(D),2^{-n\g \s})_{p,w}.
  \end{equation}

{\rm (B)} Suppose that there exists $m\in \N$ such that $\psi(x)\neq 0$ for all $x\in [0,2^{-m}]$ and the operators $\widetilde{\Psi}_{2^n}$  are such that, for any $f\in L_{p,w}(\mathcal{D})$ and $n\in\N$,
\begin{equation}\label{dopa'}
  \Vert \widetilde{\Psi}_{2^n} f\Vert_{p,w}\le C \Vert f\Vert_{p,w},
\end{equation}
where the constant $C$ does not depend on $f$ and $n$. Then
    \begin{equation}\label{R+}
       K_\g(f,Q(D),2^{-n\g \s})_{p,w}\lesssim\(\sum_{k=n+1}^\infty 2^{-\s \g \t k} \Vert Q(D)^\g \Psi_{2^{k}} f\Vert_{p,w}^\t\)^\frac1\t.
  \end{equation}
\end{theorem}

\begin{proof}
To prove inequality~\eqref{L+}, it is enough to note that by~\eqref{dopa} one has
$$
\Vert Q(D)^\g\Psi_{2^n} f\Vert_{p,w}=\Vert Q(D)^\g\Psi_{2^n} (\eta_{2^{n+1}}f)\Vert_{p,w}\le C\Vert Q(D)^\g \eta_{2^{n+1}} f\Vert_{p,w}.
$$
Thus, \eqref{L} clearly implies~\eqref{L+}.

To show~\eqref{R+}, we note that by \eqref{dopa'}, we have
\begin{equation*}
  \begin{split}
    \bigg\Vert \sum_{k=0}^n \eta(2^{-n}k)A_k(f)\bigg\Vert_{p,w}&=\bigg\Vert \sum_{k=0}^n \eta(2^{-n}k)\psi(2^{-n-m}k)(\psi(2^{-n-m}k))^{-1}A_k(f)\bigg\Vert_{p,w}\\
     &\le C\bigg\Vert \sum_{k=0}^{n+m} \psi(2^{-n-m}k)A_k(f)\bigg\Vert_{p,w},
   \end{split}
\end{equation*}
which gives
\begin{equation}\label{Hor+}
  \Vert Q(D)^\g \eta_{2^n} f\Vert_{p,w}\lesssim \Vert Q(D)^\g \Psi_{2^{n+m}} f\Vert_{p,w}.
\end{equation}
This and~\eqref{R} imply
    \begin{equation*}\label{R++}
    \begin{split}
       K_\g(f,Q(D),2^{-n\g \s})_{p,w}&\lesssim\(\sum_{k=n+1}^\infty 2^{-\s \g \t k} \Vert Q(D)^\g \Psi_{2^{k+m}} f\Vert_{p,w}^\t\)^\frac1\t\\
       &\lesssim\(\sum_{k=n+1}^\infty 2^{-\s \g \t k} \Vert Q(D)^\g \Psi_{2^{k}} f\Vert_{p,w}^\t\)^\frac1\t,
    \end{split}
  \end{equation*}
completing the proof.
\end{proof}

\begin{remark}\label{remmuk}
{\rm (i)} \emph{By Assumption~\ref{H}, condition~\eqref{dopa} can be replaced by the condition that
the sequence
$\left\{\psi\(k{2^{-n}}\)\right\}_{k\in \Z_+}$ satisfies~\eqref{eq(3.7)}.
Similarly,  condition~\eqref{dopa'} can be replaced by the condition that
the sequence
$\{\widetilde{\psi}\(k{2^{-n}}\)\}_{k\in \Z_+}$ satisfies~\eqref{eq(3.7)}.}

{\rm (ii)} \emph{If $\psi\in C^r([0,\infty)$, then both sequences $\{\psi\(k{2^{-n}}\)\}_{k\in \Z_+}$ and $\{\widetilde{\psi}\(k{2^{-n}}\)\}_{k\in \Z_+}$ satisfy \eqref{eq(3.7)} with $\ell_0=r$.}

{\rm (ii)} \emph{Inequalities~\eqref{L+} and~\eqref{R+} can be  written similarly to those in Corollary~\ref{CCCC}.}
\end{remark}

\begin{example}
\emph{Many classical Fourier means are covered by Theorem~\ref{thPsi}. In particular, these cases include the following operators
$
\Psi_n f\sim \sum_{k=0}^n \psi\(\frac kn\) A_k f:
$}

\begin{enumerate}
  \item[1)] \emph{Partial sums of Fourier series, the case $\psi(x)=\chi_{[0,1]}(x)$;}
  \item[2)] \emph{Fej\'er means that are generated by the function $\psi(x)=(1-x)_+$;}
  \item[3)] \emph{More generally, Riesz means for which $\psi(x)=(1-x^\alpha)_+^\delta$, $\alpha,\delta>0$;}
  \item[4)] \emph{Rogosinskii means that are generated by}
$$
\psi(x)=\left\{
          \begin{array}{ll}
            \cos\(\frac{\pi x}{2}\), & \hbox{$0\le x\le 1$,} \\
            0, & \hbox{$x>1$;}
          \end{array}
        \right.
$$
  \item[5)]  \emph{Jackson means, the case}
$
\psi(x)=\frac32 (1-|x|)_+ *(1-|x|)_+.
$
\end{enumerate}
\end{example}
%
%
%
%

The precise  formulation of the corresponding results in the periodic case will be given in Corollary~\ref{o-sn1}.
%
%


%



\section{General approximation processes and measures of smoothness}\label{sec2}

For a fixed positive $\l$, we consider a metric space $(X,\rho)$ with the metric $\rho : X\times X \mapsto \R_+ $ defined by
$$
\rho(f,g)=\Vert f-g\Vert_X^\l,
$$
where the functional  $\Vert \cdot \Vert_X : X\mapsto \R_+$ is such that for all $f,g\in X$ the following properties hold:

\medskip

i) $\Vert f\Vert_X=0$ if and only if $f=0$,

ii)
$
  \Vert -f\Vert_X=\Vert f\Vert_X,
$

iii)
$
  \Vert f+g\Vert_X^\l \le \Vert f\Vert_X^\l+\Vert g\Vert_X^\l.
$


\medskip

Note that the metric $\Vert \cdot \Vert_X=\rho(f,0)$ is not a norm in general since the homogeneity property is not assumed.



Let us consider the following functional, which to some extend, plays a role of a measure of smoothness (abstract modulus of smoothness)
$$
\Omega(f,\d)_X\,:\, X\times (0,\infty)\mapsto \R_+,
$$
which satisfies the following conditions: for any $f,g\in X$ and $\d>0$,
\begin{equation}\label{eq1}
  \Omega(f,\d)_X\to 0\quad\text{as}\quad \d\to +0,
\end{equation}
\begin{equation}\label{eq2}
  \Omega(f,\d)_X\le C_1\Vert f\Vert_X,
\end{equation}
\begin{equation}\label{eq3}
  \Omega(f+g,\d)_X\le C_2\(\Omega(f,\d)_X+\Omega(g,\d)_X\),
\end{equation}
\begin{equation}\label{eq4}
   \Omega(f,2\d)_X\le C_3 \Omega(f,\d)_X,
\end{equation}
where $C_j=C_j(X,\l)$, $j=1,2,3$.

\medskip

As an approximation tool, we consider the family of operators $P_n\,:\,X\mapsto X$, $n\in \N$, such that the following two properties hold:  for any $f\in X$ and $n\in \N$,

\begin{equation}\label{eq5}
  \Vert f-P_{n}(f)\Vert_X \le \Vert f-P_n(P_{2n}(f))\Vert_X
\end{equation}
\begin{equation}\label{eq6}
  \Vert f-P_n(f)\Vert_X \le C_4 \Omega\(f,n^{-1}\)_X,
\end{equation}
where $C_4=C_4(X,\l)$.

Inequality~\eqref{eq5} trivially holds when $P_n(f)$ is a best approximant to $f$ in $X$ or $P_n(f)$ is such that $P_n(P_{2n}(f))=P_n(f)$, for example, take a de la Vall\'ee Poussin--type operator or a projection operator. The second inequality is the Jackson--type theorem.

\begin{theorem}\label{le1}
Let $f\in X$ and $n\in  \N$. Then
  \begin{equation}\label{eq7}
    \Omega(P_{2^n}(f),2^{-n})_X\lesssim \Omega(f,2^{-n})_X\lesssim \(\sum_{k=n+1}^\infty \Omega(P_{2^k}(f),2^{-k})_X^\l\)^\frac1\l,
  \end{equation}
where the left-hand side inequality holds if we assume only~\eqref{eq2}, \eqref{eq3}, and~\eqref{eq6}.
\end{theorem}



Note that in the case of the Banach space $X$, a similar result for $K$-functionals and holomorphic semi-groups was obtained in~\cite[Lemmas 3.5.4 and 3.5.5]{BB}.

\begin{proof}[Proof of Theorem~\ref{le1}]
By~\eqref{eq3},
$$
\Omega(P_{2^n}(f),2^{-n})_X \lesssim \Omega(P_{2^n}(f)-f,2^{-n})_X+\Omega(f,2^{-n})_X,
$$
and the left-hand side estimate in~\eqref{eq7} follows from~\eqref{eq2} and~\eqref{eq6}.


Let us prove the right-hand side inequality.
Denote
$$
I_{2^n}:=\Vert P_{2^{n+1}}(f)-P_{2^n}(P_{2^{n+1}}(f))\Vert_X.
$$
Then by~\eqref{eq6} and~\eqref{eq4}, we have
\begin{equation}\label{eq8}
  I_{2^n}\lesssim\Omega(P_{2^{n+1}}(f),2^{-n})_X\lesssim\Omega(P_{2^{n+1}}(f),2^{-n-1})_X.
\end{equation}
At the same time, by~\eqref{eq5} we get
\begin{equation}\label{eq9}
  \begin{split}
     I_{2^n}^\l&=\Vert P_{2^{n+1}}(f)-f+f-P_{2^{n}}(P_{2^{n+1}}(f))\Vert_X^\l\\
     &\ge\Vert f-P_{2^n}(P_{2^{n+1}}(f))\Vert_X^\l-\Vert f-P_{2^{n+1}}(f)\Vert_X^\l\\
     &\ge\Vert f-P_{2^n}(f)\Vert_X^\l-\Vert f-P_{2^{n+1}}(f)\Vert_X^\l\\
     &=:E_{2^n}^\l-E_{2^{n+1}}^\l.
  \end{split}
\end{equation}
By~\eqref{eq6} and~\eqref{eq1}, $E_{2^k}\to 0$ as $k\to \infty$. Thus, \eqref{eq8} and~\eqref{eq9} imply
\begin{equation}\label{eq10}
  \begin{split}
    E_{2^n}^\l=\sum_{k=n}^\infty \(E_{2^k}^\l-E_{2^{k+1}}^\l\)\le \sum_{k=n}^\infty I_{2^k}^\l\lesssim\sum_{k=n}^\infty \Omega(P_{2^{k+1}}(f),2^{-k-1})_X^\l.
  \end{split}
\end{equation}
Then, using properties of the modulus of smoothness, namely~\eqref{eq4}, \eqref{eq3}, and \eqref{eq2}, we obtain
\begin{equation*}
  \begin{split}
     \Omega(f,2^{-n})_X^\l &\lesssim \Omega(f,2^{-n-1})_X^\l\\
     &\lesssim \(\Omega(f-P_{2^{n+1}}(f),2^{-n-1})_X^\l+\Omega(P_{2^{n+1}}(f),2^{-n-1})_X^\l\)\\
     &\lesssim \Vert f-P_{2^{n+1}}(f)\Vert_X^\l+ \Omega(P_{2^{n+1}}(f),2^{-n-1})_X^\l\\
     &= E_{2^{n+1}}^\l+ \Omega(P_{2^{n+1}}(f),2^{-n-1})_X^\l.
     \end{split}
\end{equation*}
Finally, tacking into account~\eqref{eq10},
\begin{equation*}
  \begin{split}
     \Omega(f,2^{-n})_X^\l&\lesssim \sum_{k=n}^\infty \Omega(P_{2^{k+1}}(f),2^{-k-1})_X^\l+ \Omega(P_{2^{n+1}}(f),2^{-n-1})_X^\l\\
     &\lesssim \sum_{k=n}^\infty \Omega(P_{2^{k+1}}(f),2^{-k-1})_X^\l,
   \end{split}
\end{equation*}
which is the right-hand side inequality of~\eqref{eq7}.
\end{proof}

\begin{remark}\label{remVN}
\emph{Under the conditions of Theorem~\ref{le1}, we have }
  \begin{equation}\label{eq7}
    \Omega(f,n^{-1})_X\lesssim \(\sum_{k=1}^\infty \Omega(P_{2^k n}(f),2^{-k}n^{-1})_X^\l\)^\frac1\l.
  \end{equation}
\emph{This inequality can be obtained by using a slight modification of the proof of Theorem~\ref{le1}. See also the proof of Lemma~8 in~\cite{HuLi05}.}
\emph{Similar assertions are also valid for Theorems~\ref{mainmain}, \ref{corthAB},  \ref{corthAB2},  \ref{thLP},   \ref{thPsi} as well as for the corresponding examples in Sections~\ref{sec5}--\ref{sec8}.}
\end{remark}


As a simple corollary of Theorem~\ref{le1} and Remark~\ref{remVN}, we have the following version of Jackson's inequality written in terms of measure of smoothness of
$P_{2^k n}(f)$.
\begin{corollary}
Let $f\in X$ and $n\in \N$. Then
\begin{equation*}
  \Vert f-P_{n}(f)\Vert_X\lesssim \(\sum_{k=1}^\infty \Omega(P_{2^k n}(f),2^{-k}n^{-1})_X^\l\)^\frac1\l.
\end{equation*}
\end{corollary}
\begin{remark}\label{fourier0}
\emph{
If we assume the more general  condition than (\ref{eq6}), namely,
$$
  \Vert f-P_n(f)\Vert_X \le C_4 \xi(n)  \Omega\(f,n^{-1}\)_X,
$$
where $C_4=C_4(X,\l)$ and $\xi$ is a positive non-decreasing function on $[1,\infty),$
then repeating the  proof of Theorem~\ref{le1} gives the following estimates
 \begin{equation}\label{eq7--}
\xi^{-1}(2^n)    \Omega(P_{2^n}(f),2^{-n})_X\lesssim \Omega(f,2^{-n})_X\lesssim \(\sum_{k=n+1}^\infty \xi^\l(2^k) \Omega(P_{2^k}(f),2^{-k})_X^\l\)^\frac1\l
\end{equation}
and
\begin{equation*}
  \Vert f-P_{2^n}(f)\Vert_X\lesssim \(\sum_{k=n+1}^\infty \xi^\l(2^k) \Omega(P_{2^k}(f),2^{-k})_X^\l\)^\frac1\l.
\end{equation*}
A typical example when Remark~\ref{fourier0} can be applied is considering the partial sums of Fourier series $P_n(f)=S_n(f)$ in the case $X=L_p(\T)$, $p=1,\infty$, and $\xi(t)=\log (t+1)$; for details see Corollary~\ref{fourier}.
}
\end{remark}

In what follows, we say that  $\w\,:\, \R_+\to\R_+$ is the modulus a continuity if $\omega$ is a positive  non-decreasing function, $\w(0)=0$, and $\w(x+y)\le \w(x)+\w(y)$ for any $x,y\in \R_+$.

\begin{corollary}\label{le1Optt1}
For any modulus of continuity $\w$ such that
\begin{equation}\label{Omeg}
 \sum_{k=n}^\infty \w(2^{-k})\lesssim \w(2^{-n}),
\end{equation}
the following assertions are equivalent:
\begin{enumerate}
  \item $\Omega(P_{2^n}(f),2^{-n})_X\lesssim \w(2^{-n})$,

  \item $\Omega(f,2^{-n})_X\lesssim \w(2^{-n})$.
\end{enumerate}
\end{corollary}

\begin{proof}
The proof follows from~\eqref{eq7} and the simple fact that~\eqref{Omeg} is equivalent to
\begin{equation}\label{Omegla}
 \(\sum_{k=n}^\infty \w(2^{-k})^\l\)^\frac1\l\lesssim \w(2^{-n})\quad\text{for any}\quad \l>0,
\end{equation}
see, e.g.,~\cite{Ti04}.
\end{proof}



For a given modulus of continuity $\w$, we define the function class
$$
\Xi_\w=\left\{f\in X\,:\, \Omega(f,\d)_X\asymp \w(\d),\quad \d\to 0\right\}.
$$

The next corollary provides sharpness of Theorem~\ref{le1}.

\begin{corollary}\label{le1Optt2}
 Let $f\in \Xi_\w$ and $\w$ satisfy~\eqref{Omeg}. Then, for large enough $n\in \N$,
  \begin{equation}\label{XXXZZZ1}
    \Omega(f,2^{-n})_X\asymp \Omega(P_{2^n}(f),2^{-n})_X \asymp \(\sum_{k=n+1}^\infty \Omega(P_{2^k}(f),2^{-k})_X^\l\)^\frac1\l.
  \end{equation}

\end{corollary}

\begin{proof}
First, we prove that
\begin{equation}\label{XXXZZZ1++++}
  \Omega(f,2^{-n})_X\asymp \Omega(P_{2^n}(f),2^{-n})_X.
\end{equation}
The part $\,\gtrsim\,$ in~\eqref{XXXZZZ1++++} is given by~\eqref{eq7}. To show the part $\,\lesssim\,$, we note that by~\eqref{Omeg} and monotonicity of $\w$, for any $m<n$, we have
\begin{equation*}\label{XXXZZZ1++++-}
  \w(2^{-n+m}) \gtrsim \sum_{k=n-m}^\infty \w(2^{-k}) \gtrsim \sum_{k=n-m}^n \w(2^{-k}) \gtrsim (m+1) \w(2^{-n}).
\end{equation*}
Then,  taking onto account~\eqref{eq4}, \eqref{eq3}, \eqref{eq2}, and~\eqref{eq6} and choosing large enough $m\in \N$, we derive
\begin{equation*}\label{eqqqq2}
  \begin{split}
      \Omega(P_{2^n}(f),2^{-n})_X^\l &\ge C_3^{-m \l}\Omega(P_{2^n}(f),2^{-n+m})_X^\l \\
      &\ge C_3^{-m \l}\(C_2^{-\l}\Omega(f,2^{-n+m})_X^\l -  \Omega(f-P_{2^n}(f),2^{-n-m})_X^\l\)\\
      &\ge C_3^{-m \l}\(C_2^{-\l}\Omega(f,2^{-n+m})_X^\l -  C_1^\l\Vert f-P_{2^n}(f)\Vert_X^\l\)\\
      &\ge C_3^{-m \l}\(C_2^{-\l} \Omega(f,2^{-n+m})_X^\l -  (C_1C_4)^\l\Omega(f,2^{-n})_X^\l\)\\
      &\ge C_3^{-m \l}\(c' \w(2^{-n+m})^\l-c''\w(2^{-n})^\l  \)\\
      & \ge C_3^{-m \l}\(c' (m+1)^\l-c''\)\w(2^{-n})^\l\\
      &\gtrsim \w(2^{-n})^\l\gtrsim  \Omega(f,2^{-n})_X^\l.
  \end{split}
\end{equation*}

To prove the second equivalence in~\eqref{XXXZZZ1}, we note the part $\,\lesssim\,$ follows  from the right-hand side inequality of~\eqref{eq7} and~\eqref{XXXZZZ1++++} while the part $\,\gtrsim\,$ follows from~\eqref{Omegla}, the left-hand side inequality in~\eqref{eq7}, and~\eqref{XXXZZZ1++++},
\begin{equation*}
  \(\sum_{k=n+1}^\infty \Omega(P_{2^k}(f),2^{-k})_X^\l\)^\frac1\l\lesssim \(\sum_{k=n}^\infty \w(2^{-k})^\l\)^\frac1\l\lesssim \w(2^{-n}) \lesssim \Omega(P_{2^n}(f),2^{-n})_X.
\end{equation*}
\end{proof}

\begin{remark}\label{ZZZXXX}
\emph{Corollaries~\ref{le1Optt1} and~\ref{le1Optt2} imply that if $\w(\d)=\d^\a$, $\a>0$, then, for any $f\in X$ and $n\in \N$, we have
$$
\Omega(f,2^{-n})_X \lesssim \w(2^{-n})\quad \text{iff}\quad \Omega(P_{2^n}(f),2^{-n})_X \lesssim \w(2^{-n}).
$$
If, in addition, $f \in \Xi_\w$, then
$$
\Omega(f,2^{-n})_X \asymp \Omega(P_{2^n}(f),2^{-n})_X \asymp \w(2^{-n}).
$$}
%
%
\end{remark}

The results of Remark~\ref{ZZZXXX} can be extended to Besov-type spaces.

For a given modulus of smoothness $\Omega$, $s>0$, and $0<q\le \infty$, we define the Besov-type space as follows:
  \begin{equation}\label{Besov}
B_{X,q}^s=\left\{f\in X: |f|_{B_{X,q}^s}=\left( \int_0^1 \big( t^{-s} \Omega(f,t)_X)^q\frac{dt}t\right)^{\frac1q}<\infty\right\}
  \end{equation}
with the usual modification in the case $q=\infty$.

We have the following characterization of $B_{X,q}^s$. 
\begin{corollary}\label{cor2.4}
Let $s>0$ and $0<q\le \infty$. We have 
$$
|f|_{B_{X,q}^s}\asymp \left( \sum_{k=1}^\infty 2^{s q k} \Omega\(P_{2^k}(f),2^{-k}\)_X^q\right)^{\frac1q}. 
$$
\end{corollary}

\begin{proof}
The proof easily follows from Theorem \ref{le1} and the Hardy-type inequality
$$
\sum_{\nu=n}^\infty 2^{\nu s}\(\sum_{k=\nu}^\infty A_k\)^q \asymp \sum_{\nu=n}^\infty 2^{\nu s} A_\nu^q,
$$
where $A_\nu \ge 0$ and $s,q>0$.
\end{proof}

%

\section{Smoothness of approximation processes on $\T^d$}\label{sec5}

\subsection{Smoothness of best approximants.}

In this subsection, we give analogues of Theorems~\ref{le1} and~\ref{corthAB} for best trigonometric approximants in $L_p(\T^d)$ spaces.
We recall some basic notations. Denote the set of all trigonometric polynomials of degree at most $n$ by
$$
\mathcal{T}_n=\text{span}\,\{e^{i(k,x)}:\,\vert k\vert\le n\},
$$
where $\vert k\vert   = (k^2_1 +\dots + k^2_d)^{1/2}.$
The best approximation by trigonometric polynomials is given by
$$
E_n  (f)_{L_p(\T^d)} =\inf\,\{\Vert  f-\varphi  \Vert
_{L_p(\T^d)}:\varphi  \in \mathcal{T}_n\}.
$$
As above, by $P_n(f)$ we denote the best approximant of a function $f$ in $L_p(\T^d)$, that is,
$$
\Vert f-P_n(f)\Vert_{L_p({\T}^d)}=E_n  (f)_{L_p({\T}^d)},
$$
where $P_n(f)\in \mathcal{T}_n$.

In what follows, we will use the well-known Jackson type inequality, see, e.g.,~\cite{Ti} and~\cite{SO}:   
\begin{equation}\label{JacksonSO}
E_n(f)_{L_p({\T}^d)}\le C\w_r\(f,\frac1n\)_{L_p({\T}^d)},\quad f\in L_p(\T^d),\quad 0<p\le\infty,\quad r\in \N,
\end{equation}
where $\w_r(f,h)_p$ is the classical modulus of smoothness,
$$
\w_r(f,\d)_p=\sup_{|h|<\d} \Vert \Delta_h^r f\Vert_{L_p(\T^d)},
$$
$$
\Delta_h f(x)= f(x+h) -f(x),\quad \Delta_h^r =
\Delta_h \Delta_h^{r-1},\quad h\in \R^d,\quad d\ge 1,
$$
and the constant $C$ does not depend on $f$ and $n$.

We will also need the following Stechkin-Nikolskii-type inequality
(see~\cite[Theorem 3.2]{KT}), which states that, for any $n\in \N$ and $0<\d\le\pi/ n$,
\begin{equation}\label{eq+++}
 \Vert T_n\Vert_{\dot W_p^r(\T^d)}\asymp \d^{-r}\w_r (T_n,\d)_{L_p(\T^d)},\quad T_n\in \mathcal{T}_n,\quad 0< p\le\infty,\quad r\in \N,
\end{equation}
where the constants in this equivalence are independent of $T_n$ and $\d$.
Here the homogeneous Sobolev norm is given by
$$
\Vert f \Vert_{\dot W_p^{r}(\T^d)}=\sum_{|\nu|_1=r}\Vert D^\nu f\Vert_{L_p(\T^d)}.
$$

Using Theorem~\ref{le1} with  $X=L_p(\T^d)$, $0<p\le\infty$, and $\Omega(f,\delta)_X=\omega_r(f,\delta)_{L_p(\T^d)}$ for some $r\in \N$, one can easily verify  that properties \eqref{eq1}--\eqref{eq6} are valid.  Therefore,  applying Stechkin-Nikolskii-type inequality~\eqref{eq+++}, we obtain the following result.
\begin{theorem}\label{PrG}
  Let $f\in L_p({\T^d})$, $0<p\le \infty$, and $r\in \N$. Then
  \begin{equation}\label{---eq7Td}
    2^{-nr}\Vert P_{2^n}(f)\Vert_{\dot W_p^r(\T^d)}\lesssim \omega_r(f,2^{-n})_{L_p(\T^d)}\lesssim  \(\sum_{k=n+1}^\infty 2^{-kr\lambda}\Vert P_{2^k}(f)\Vert_{\dot W_p^r(\T^d)}^\l\)^\frac1\l,
  \end{equation}
 where $\l=\min(p,1)$. 
\end{theorem}

%

The above theorem can be also formulated in terms of the fractional smoothness. For this, we recall the following assertion from~\cite[Corollary 3.1]{KT}:
{\it Let $0<p\le\infty$, $\a>0$, $n\in\mathbb{N}$, and
$0<\d\le \pi/ n$. Then, for any
$T_n\in\mathcal{T}_n$, we have}
\begin{equation}\label{ineqNS3cor}
\sup_{\xi\in \R^d,\,|\xi|=1}\bigg\Vert \(\frac{\partial}{\partial\xi}\)^{\a} T_n\bigg\Vert_{L_p({\T}^d)}
\asymp \d^{-\a}\w_\a (T_n,\d)_{L_p({\T}^d)},
\end{equation}
where the fractional modulus of smoothness $\w_\a (f,\d)_{L_p({\T}^d)}$   is given by
\begin{equation*}\label{def-mod+}
\omega_{\a} (f,\delta)_{L_p(\T^d)}
 =\sup_{|h|\le \delta }
 \bigg\|
\sum\limits_{\nu=0}^\infty(-1)^{\nu}
\binom{\a}{\nu} f\,\big(\cdot+(\a-\nu) h\big) \bigg\|_{L_p(\T^d)}, 
\end{equation*}
and $\binom{\a}{\nu}=\frac{\a (\a-1)\dots (\a-\nu+1)}{\nu!}$,\quad
$\binom{\a}{0}=1$, see~\cite{PST}.


Our next goal is to obtain a sharp version of~\eqref{---eq7Td} in the case $1<p<\infty$.
For this, we use Theorem~\ref{corthAB} with $G_n=\mathcal{T}_n$, $X=L_p(\T^d)$, and $Y=H_p^{\a}(\T^d)$, where
$$
H_p^{\a}(\T^d)=\{g\in L_p(\T^d)\,:\, \Vert g \Vert_{\dot H_p^{\a}(\T^d)}=\Vert (-\Delta)^{\a/2} g\Vert_{L_p(\T^d)}<\infty\}
$$
is the fractional Sobolev space. Recall that
\begin{equation}\label{KTD}
  K\(f,t^{\a},L_p(\T^d);H_p^{\a}(\T^d)\)= \inf \left\{\Vert f-g\Vert_{L_p(\T^d)} + t^{\a}\Vert g \Vert_{\dot H_p^{\a}(\T^d)}\, :\, g\in H_p^{\a}(\T^d) \right\}
\end{equation}
and
\begin{equation}\label{RTD}
R\(f,t^{\a};L_p(\T^d),\mathcal{T}_{[1/t]}\)= \inf \left\{\Vert f-T\Vert_{L_p(\T^d)} + t^{\a}\Vert T\Vert_{\dot H_p^{\a}(\T^d)}\,:T\,\in \mathcal{T}_{[1/t]} \right\}
\end{equation}
(cf.~\eqref{K1.1} and~\eqref{R1.1}).
For any $f\in L_p(\T^d)$, $1<p<\infty$, and $\a>0$ we have (see, e.g.,~\cite{KT})
$$
K(f,t^{\a};L_p(\T^d),H_p^\a(\T^d))\asymp R(f,t^{\a};L_p(\T^d),\mathcal{T}_{[1/t]})\asymp \omega_\a(f,t)_{L_p(\T^d)},
$$
which, in particular, implies~\eqref{r1--}.

Jackson and Bernstein inequalities~\eqref{K1.2} and~\eqref{K1.3} are given by~\eqref{JacksonSO} and the following inequality, see, e.g.,~\cite{Wil},
$$
\Vert  (-\Delta  )^{\a/2 } T_n\Vert_{L_p(\T^d)}\lesssim n^\a \Vert  T_n\Vert_{L_p(\T^d)},\quad T_n\in\mathcal{T}_n,\quad 1<p<\infty,\quad \a>0.
$$

%
%




Thus, Theorem~\ref{corthAB} implies the following result.

\begin{theorem}\label{propAB}
  Let $f\in L_p(\T^d)$, $1<p< \infty$, and $\a>0$. Then
  \begin{equation*}
  \begin{split}
          \(\sum_{k=n+1}^\infty 2^{-k\a\tau}\Vert (-\Delta)^{\a/2}P_{2^k}(f)\Vert_{L_p(\T^d)}^\tau\)^\frac1\tau&\lesssim \omega_\a(f,2^{-n})_{L_p(\T^d)}\\
    &\lesssim \(\sum_{k=n+1}^\infty 2^{-k\a\t}\Vert (-\Delta)^{\a/2}P_{2^k}(f)\Vert_{L_p(\T^d)}^\t\)^\frac1\t,
  \end{split}
  \end{equation*}
 where $\tau=\max(2,p)$ and $\t=\min(2,p)$.
\end{theorem}

%


\subsection{The case of Fourier multiplier operators.}

In this subsection, we give an analogue of Theorem~\ref{thPsi} in the case $\mathcal{D}=\T^d$. We start by recalling the multiplier theorem (Assumption~\ref{H}) and the Littlewood-Paley-type theorem in $L_p(\T^d)$ for $1<p<\infty$.

Concerning Assumption~\ref{H}, the well-known Mikhlin-H\"ormander multiplier theorem (see~\cite[p.~224]{Gr}) states that the condition
\begin{equation}\label{MiHo}
  |\Delta  ^{\b_1}_{e_1} \dots \Delta  ^{\b_d}_{e_d} m(k_1,\dots,k_d)|  \le A |k|^{-\vert
\b \vert  },\quad \vert  \b  \vert  \equiv \b_1 + \dots + \b_d <[
d/2] + 1,
\end{equation}
where $\Delta  _{e_i}m(k_1,\dots,k_i,k_d) =
m(k_1,\dots,k_i+1,\dots,k_d) - m(k_1,\dots,k_i,\dots,k_d),$ implies
$$
\Vert  T_m f\Vert_{L_p(\T^d)} \le C(A,p) \Vert  f\Vert  _{L_p(\T^d)},
$$
where
$$
(T_m f)^\wedge (k) = m(k)\wh f(k)
$$
and
$ \wh f(k) = \frac1{(2\pi)^d}\int_{\T^d} f(y) e^{-i(k,y)} dy.$

We define the de~la~Vall\'ee Poussin-type multiplier operator by
$$
(\eta_n f)^\wedge(k) = \eta  \Big(\frac{\vert  k\vert  }{n}\Big)\wh
f(k)
$$
and similarly to~\eqref{LPVP}, we set
$$
\theta_0(f) = \eta_1f \q\text{and} \q \theta_j (f)= \eta_{2^j}f
- \eta_{2^{j-1}} f \q\text{for} \q j\ge 1.
$$

An analogue of the Littlewood-Paley theorem in the case $\mathcal{D}=\T^d$ is given by the following two inequalities, see, e.g.,~\cite[Theorem 4.1]{DDT} or \cite[Ch.~6]{GrII}:
for $f\in L_p(\T^d)$, $1<p<\infty,$ and $\a>0$, we have
$$
\Big\Vert  \Big\{\sum^\infty  _{j=0} (\theta
 _j(f))^2\Big\}^{1/2}\Big\Vert_{L_p(\T^d)} \asymp \Vert  f\Vert_{L_p(\T^d)}
$$
and
$$
\Big\Vert  \Big\{\sum^\infty_{j=1} 2^{2j\a} \big(\theta_j
(f)\big)^2\Big\}^{1/2} \Big\Vert_{L_p(\T^d)} \asymp\Vert
(-\Delta)^{\a/2}  f\Vert_{L_p(\T^d)}.
$$


Let us consider the Fourier means given by
$$
\Psi_n f(x)=\sum_{k\in \Z^d} \psi\(\frac{k}{n}\) \widehat{f}(k) e^{i(k,x)},
$$
$$
\widetilde{\Psi}_n f(x)=\sum_{k\in \Z^d} \widetilde{\psi}\(\frac{k}{n}\) \widehat{f}(k) e^{i(k,x)},\quad \widetilde{\psi}(\xi)=\frac{\eta(|\xi|)}{\psi(2^{-m}\xi)},
$$
where the function $\psi:\R^d\to \C$ is such that $\supp \psi\subset [-1,1]^d$ and for some $m\in\Z_+$, $\psi(x)\neq 0$ for all $x\in [-2^{-m},2^{-m}]^d$.

We derive the following analogue of Theorem~\ref{thPsi}  in the case $\mathcal{D}=\T^d$.

\begin{theorem}\label{propVPTd}
  Let $f\in L_p(\T^d)$, $1<p< \infty$, $n\in\N$, $\a>0$, $\tau=\max(2,p)$, and $\t=\min(2,p)$.

 {\rm (A)}  If $\{\Psi_{2^k}\}$ are uniformly bounded operators in $L_p(\T^d)$, then
  \begin{equation*}\label{eq7-Td}
    \(\sum_{k=n+1}^\infty 2^{-k\a\tau}\Vert (-\Delta)^{\a/2}\Psi_{2^k}f\Vert_{L_p(\T^d)}^\tau\)^\frac1\tau\lesssim \omega_\a(f,2^{-n})_{L_p(\T^d)}.
  \end{equation*}

  {\rm (B)}  If $\{\widetilde{\Psi}_{2^k}\}$ are uniformly bounded operators in $L_p(\T^d)$, then
    \begin{equation*}\label{eq7+Td} \omega_\a(f,2^{-n})_{L_p(\T^d)}\lesssim  \(\sum_{k=n+1}^\infty 2^{-k\a\t}\Vert (-\Delta)^{\a/2}\Psi_{2^k}f\Vert_{L_p(\T^d)}^\t\)^\frac1\t.
  \end{equation*}
\end{theorem}

\begin{remark}\label{rem5.1}
{\rm (i)} \emph{Note that if $\psi \in A(\R^d)=\{f\,:\, f=\widehat{g},\,\, g\in L_1(\R^d)\}$ (the Wiener class of absolutely convergent Fourier integrals), then
the operators $\{\Psi_{n}\}$ are uniformly bounded in $L_p(\T^d)$ for all $1\le p\le\infty$, see, e.g.,~\cite[Ch. VII]{SW}. Various useful conditions to insure that $\psi\in A(\R^d)$ can be found in the survey~\cite{LST}, see also~\cite[Ch.~4 and 6]{TB}.}

{\rm (ii)} \emph{Concerning the uniform boundedness of $\{\widetilde{\Psi}_n\}$, one can use following version of $\frac1f$-Wiener  theorem
(see \cite[p.102]{Lo}):
Let $f\in A(\R^d)$. If $f(x)\neq 0$ on a closed bounded set $V\subset\R^d$, then
$\frac1{f(x)}$ is extendable to a function in $A(\R^d)$, i.e., there exists a function $g\in A(\R^d)$
such that $f(x)\equiv g(x)$ on $V$.}

{\rm (iii)}\emph{ To verify the uniform boundedness of $\{\Psi_n\}$ and $\{\widetilde{\Psi}_n\}$ in $L_p(\T^d)$ for $1<p<\infty$, one can use  the Mikhlin-H\"ormander multiplier condition~\eqref{MiHo}, which is less restrictive than the conditions given in parts ${\rm (i)}$ and ${\rm (ii)}$ of this remark.}

{\rm (iv)} \emph{Under conditions of Theorem~\ref{propVPTd}, we have that for any $f\in H_p^\b (\T^d)$, $\b>0$,
  \begin{equation*}\label{eq7-Td}
    \(\sum_{k=n+1}^\infty 2^{-k\a\tau}\Vert (-\Delta)^{{(\a+\b)}/2}\Psi_{2^k}f\Vert_{L_p(\T^d)}^\tau\)^\frac1\tau\lesssim \omega_\a((-\Delta)^{\b/2}f,2^{-n})_{L_p(\T^d)}
  \end{equation*}
and
    \begin{equation*}\label{eq7+Td} \omega_\a((-\Delta)^{\b/2}f,2^{-n})_{L_p(\T^d)}\lesssim  \(\sum_{k=n+1}^\infty 2^{-k\a\t}\Vert (-\Delta)^{{(\a+\b)}/2}\Psi_{2^k}f\Vert_{L_p(\T^d)}^\t\)^\frac1\t.
  \end{equation*}}
\end{remark}



As examples, let us consider the following approximation processes:

1) the $\ell_q$-partial Fourier sums
$$
S_{n;q} f(x)=\sum_{\Vert k\Vert_{\ell_q}\le n} \widehat{f}(k) e^{i(k,x)},\quad 1\le q\le \infty;
$$

2) the de~la~Vall\'ee Poussin-type means
$$
\eta_{n}f(x)=\sum_{k\in \Z^d} \eta\(\frac{|k|}n\) \widehat{f}(k) e^{i(k,x)};
$$

3) the Riesz spherical  means
$$
R_n^{\beta,\d} f(x)=\sum_{|k|\le n} \bigg(1- \bigg(\frac{\left|k\right|}n\bigg)^\b\bigg)_+^\delta \widehat{f}(k) e^{i(k,x)},\quad \b,\d>0.
$$


\begin{corollary}\label{o-sn1}
Let $f \in L_p(\T^d)$, $1 < p < \infty$, $\a>0$,  $\tau = \max(2,p)$, and $\theta = \min(2,p)$.
 Then
\begin{equation}\label{rr++++}
 \( \sum\limits_{k =n+1}^{\infty} 2^{- k \a \tau}
\Vert (-\Delta)^{\a/2}T_{2^k}f \Vert_p^{\tau}
\)^{\frac{1}{\tau}} \lesssim \omega_{\a} \Big( f,
\frac{1}{2^n} \Big)_p \lesssim  \( \sum\limits_{k =n+1}^{\infty} 2^{- k \a \t}
\Vert (-\Delta)^{\a/2}T_{2^k}f \Vert_p^{\t}
\)^{\frac{1}{\t}},
\end{equation}
where $T_{2^k} f=S_{2^k;q} f$ with $q=1,\infty$, $\eta_{2^k}f$,  or $R_{2^k}^{\b,\d} f$ with $\d>(d-1)/2$.
\end{corollary}

\begin{proof}
It is enough to note that these means are uniformly bounded in $L_p(\T^d)$, $1<p<\infty$, see, e.g.,~\cite[Ch. VII]{SW} and~\cite{We}, and to apply the
Mikhlin-H\"ormander multiplier condition to show that the corresponding inverse operators $\{\widetilde{\Psi}_n\}$ are also uniformly bounded in $L_p(\T^d)$.
\end{proof}

\begin{remark} \emph{In the univariate case of the Fej\'er means $T_{2^k} f=R_{2^k}^{1,1} f$, the right-hand side of inequality~\eqref{rr++++} was obtained earlier by Zhuk and Natanson in~\cite{Zh}.}
%
\end{remark}

Note that for $\a\in \N$ and $1<p<\infty$ inequality~\eqref{rr++++} can be equivalently  written as follows
$$
 \( \sum\limits_{k =n+1}^{\infty} 2^{- k \a \tau}
\Vert T_{2^k}f \Vert_{\dot W_p^\a(\T^d)}^{\tau}
\)^{\frac{1}{\tau}} \lesssim \omega_{\a} \Big( f,
\frac{1}{2^n} \Big)_p \lesssim  \( \sum\limits_{k =n+1}^{\infty} 2^{- k \a \t}
\Vert T_{2^k}f \Vert_{\dot W_p^\a(\T^d)}^{\t}
\)^{\frac{1}{\t}}.
$$
We  give its analogue for the cases $p=1,\,\infty$.
\begin{corollary}\label{fourier}
Let $f \in L_p(\T^d)$, $p =1, \infty$, and $\a\in \N$.
 Then
\begin{equation}\label{fourier1}
 2^{- n \a } \xi^{-1}_q(2^n)
\Vert S_{2^n;q} f \Vert_{\dot W_p^\a(\T^d)} \lesssim \omega_{\a} \Big( f,
\frac{1}{2^n} \Big)_p\lesssim \sum\limits_{k =n+1}^{\infty} 2^{- k \a } \xi_q(2^k)
\Vert S_{2^k;q} f \Vert_{\dot W_p^\a(\T^d)},
\end{equation}
where
$$
\xi_q(t)=\left\{
           \begin{array}{ll}
             \log^d(t+1), & \hbox{$q=1,\infty$,} \\
             t^{\frac{d-1}{2}}, & \hbox{$1<q<\infty$, $q\neq 1$,}
           \end{array}
         \right.
$$
and
\begin{equation}\label{fourier2}
 2^{- n \a }
\Vert T_{2^n}f \Vert_{\dot W_p^\a(\T^d)} \lesssim \omega_{\a} \Big( f,
\frac{1}{2^n} \Big)_p\lesssim \sum\limits_{k =n+1}^{\infty} 2^{- k \a}
\Vert T_{2^k}f \Vert_{\dot W_p^\a(\T^d)},
\end{equation}
where $T_{2^k} f=\eta_{2^k}f$  or $R_{2^k}^{\b,\d} f$ with $\d>(d-1)/2$.
\end{corollary}

\begin{proof}
Estimates~\eqref{fourier1} follow from Remark~\ref{fourier0} with $\xi(t)=\xi_q(t)$ since 
$$
\Vert f-S_{n;q}f\Vert_{L_p(\T^d)}\lesssim \Vert S_{n;q} \Vert_{L_1\to L_1} E_{cn}(f)_{L_p(\T^d)}\lesssim \xi_q(n) \w_\a (f,n^{-1})_{L_p(\T^d)}.
$$
For calculation of $\xi_q(t)$ see, e.g.,~\cite{Li} and~\cite{Dy} for the case $1<q<\infty$ and~\cite[Sec.~9.2]{TB}, \cite{KL} for the case $q=1,\infty$.

The proof of~\eqref{fourier2} for $T_{2^k} f=\eta_{2^k}f$ follows from Theorem~\ref{le1} and the uniform boundedness of the de la Val\'ee Poussin means in $L_1(\T^d)$,  see also Remark~\ref{rem5.1}. The case $T_{2^k} f=R_{2^k}^{\b,\d} f$ can be proved similarly using the uniform boundedness of $R_{2^k}^{\b,\d}$, see, e.g.,~\cite[Ch. VII]{SW}, the inequality $\Vert f-R_{2^k}^{\b,\d} f\Vert_{L_p(\T^d)}\lesssim \omega_{\a} (f, {2^{-n}})_p$, see~\cite{Wil}, and applying the same arguments as in  the proof of~\eqref{R+}.
\end{proof}

\subsection{Inequalities in the Hardy spaces $H_p(D)$, $0<p\le 1$}

For simplicity, we only consider the analytic Hardy spaces on the unit disc $D=\{z\in\C\,:\, |z|<1\}$.
By definition, an analytic function $f$ on $D$
belongs to the space $H_p=H_p(D)$ if
\begin{equation*}
\Vert f\Vert_{H_p}=\sup_{{0<\rho<1}}\left(\int_{0}^{2\pi}|f(\rho
e^{it})|^p dt\right)^{\frac1p}<\infty.
\end{equation*}\index{\bigskip\textbf{Spaces}!$H_p(D)$}\label{HPD}

Set
$$
\eta_n f(x)=\sum_{k=0}^{n} \eta\(\frac kn\)c_k e^{ikx},
$$
where
$c_k=c_k(f)$ are the Taylor coefficients of~$f$. Then,
the realization result is given as follows (see~\cite[Sec.~11]{KT}):
\begin{equation*}\label{KamodHardy}
  \Vert f-\eta_{2^n}f\Vert_{H_p}+2^{-\a n}\Vert (\eta_{2^n}f)^{(\a)}\Vert_{H_p}\asymp \w_\a(f,2^{-n})_{H_p}.
\end{equation*}

Using the scheme of the proof of Theorem~\ref{thLP} and the Littlewood-Paley  theorem in the Hardy spaces $H_p(D)$, $0<p\le 1$, see, e.g.,~\cite[Ch.~6]{GrII}, we obtain the following result.

\begin{theorem}\label{thLPHardy}
  Let $f\in H_{p}({D})$, $0<p\le 1$, $\a\in \N\cup (1/p-1,\infty)$, $n\in \N$. Then
  \begin{equation}\label{LHp}
    \(\sum_{k=n+1}^\infty 2^{-2\a k} \Vert (\eta_{2^k} f)^{(\a)}\Vert_{H_p}^2\)^\frac12 \lesssim \w_\a(f,2^{-n})_{H_p}
  \end{equation}
and
    \begin{equation}\label{RHP}
   \w_\a(f,2^{-n})_{H_p}\lesssim\(\sum_{k=n+1}^\infty 2^{-\a p k} \Vert (\eta_{2^k} f)^{(\a)}\Vert_{H_p}^p\)^\frac1p.
  \end{equation}
\end{theorem}

\begin{remark}
{\rm (i)} \emph{Note that the restriction $\a>1/p-1$ is needed to correctly define the modulus of smoothness $\w_\a(f,\delta)_{H_p}$.}

{\rm (ii)}  \emph{Inequalities~\eqref{LHp} and~\eqref{RHP} are also valid if we replace the de la Vall\'ee Poussin means $\eta_{2^k} f$ by the corresponding  means $\Psi_{2^k} f$ with the properties similar to those indicated in Theorem~\ref{thPsi}.}

{\rm (iii)} \emph{Inequality~\eqref{RHP} also follows from Theorem~\ref{le1} and the Stechkin-Nikolskii inequality~\eqref{ineqNS3cor}.}
\end{remark}

\subsection{Approximation in smooth function spaces}

We will say that $f\in {\rm Lip}(\a,p)(\T)$, $0<p\le\infty$, $\a>0$, if $f\in L_p(\T)$ and
\begin{equation*}\label{eqI4}
  \Vert f\Vert_{{\rm Lip}(\a,p)}=\Vert f\Vert_{L_p(\T)}+|f|_{{\rm Lip}(\a,p)}<\infty,
\end{equation*}
where
\begin{equation*}
  |f|_{{\rm Lip}(\a,p)}=\sup_{h>0}\frac{\Vert \D_h^r f\Vert_{L_p(\T)}}{h^\a}=\sup_{h>0}\frac{\w_r(f,h)_p}{h^\a},\quad r=[\a]+1.
\end{equation*}

Let $0<p\le\infty$, $0<\a< \ell$, and $\ell,n\in\N$. The best approximation in ${\rm Lip}(\a,p)(\T)$ and the modulus of smoothness are given by
$$
E_n(f)_{{\rm Lip}(\a,p)}=\inf_{T\in \mathcal{T}_n}\Vert f-T\Vert_{{\rm Lip}(\a,p)}
$$
and
$$
\vartheta_{\ell,\a}(f,\d)_p=\sup_{0<h\le \d}\frac{\w_\ell(f,h)_p}{h^\a}.
$$

In light of the Jackson  inequality (see~\cite{KP})
\begin{equation*}\label{eqR2}
  E_n(f)_{{\rm Lip}(\a,p)}\lesssim \vartheta_{\ell,\a}\(f,\frac1n\)_p,\quad n\in \N,
\end{equation*}
by~\eqref{eq+++}, the realization result can be written as follows
\begin{equation}\label{HolSN}
 \vartheta_{\ell,\a}(f,\d)_p\asymp \Vert f-T_n\Vert_{{\rm Lip}(\a,p)}+ \d^{\ell-\a}\Vert T_{n}^{(\ell)}\Vert_{L_p(\T)},\quad n=[1/\d],
\end{equation}
where $T_n\in \mathcal{T}_n$ is such that $E_n(f)_{{\rm Lip}(\a,p)}=\Vert f-T_n\Vert_{{\rm Lip}(\a,p)}$.


Therefore,  making use of Theorem~\ref{le1} with $X={\rm Lip}(\a,p)$ and $\Omega(f,\delta)_X=\vartheta_{\ell,\a}(f,\delta)_{p}$, $\a<\ell$, $\ell\in \N$, and~\eqref{HolSN}, we obtain the following result.
\begin{theorem}\label{PrG}
  Let $f\in {\rm Lip}(\a,p)$, $0<p\le \infty$, $\ell\in \N$, $0<\a<\ell$, and $\l=\min(p,1)$. Then
  \begin{equation}\label{eq7Td++}
    2^{-n(\ell-\a)}\Vert T_{2^n}^{(\ell)}\Vert_{L_p(\T)}\lesssim \vartheta_{\ell,\a}(f,2^{-n})_{p}\lesssim  \(\sum_{k=n+1}^\infty 2^{-k (\ell-\a) \l}\Vert T_{2^k}^{(\ell)}\Vert_{L_p(\T)}^\l\)^\frac1\l,
  \end{equation}
 where  $T_{2^k}\in \mathcal{T}_{2^k}$ is the best approximant of $f$ in ${\rm Lip}(\a,p)$.
\end{theorem}


In view of Theorem~\ref{corthAB}, we  sharpen~\eqref{eq7Td++} for $1<p<\infty$ as follows.

\begin{theorem}\label{PrG2}
  Let $f\in {\rm Lip}(\a,p)$, $1<p< \infty$, $\ell\in \N$,  $0<\a<\ell$, and $\tau = \max(2,p)$, $\theta = \min(2,p)$. Then
  \begin{equation*}\label{eq7Td}
    \(\sum_{k=n+1}^\infty 2^{-k (\ell-\a) \tau}\Vert T_{2^k}^{(\ell)}\Vert_{L_p(\T)}^\tau\)^\frac1\tau\lesssim \vartheta_{\ell,\a}(f,2^{-n})_{p}\lesssim  \(\sum_{k=n+1}^\infty 2^{-k (\ell-\a) \t}\Vert T_{2^k}^{(\ell)}\Vert_{L_p(\T)}^\t\)^\frac1\t,
  \end{equation*}
 where $T_{2^k}\in \mathcal{T}_{2^k}$ is the best approximant of $f$ in ${\rm Lip}(\a,p)$. 
\end{theorem}

\begin{remark}
\emph{Using the well-known facts about simultaneous approximation of functions and their derivatives in $L_p(\T)$, see, e.g.~\cite{CzFr} and~\cite[Ch.7, Theorem~2.7]{DeLo}, it is not difficult to obtain analogues of Theorems~\ref{PrG} and~\ref{PrG2} in the Sobolev spaces $W_p^r (\T)$, $1\le p\le \infty$, and $r\in \N$, cf. Remark~\ref{rem5.1} (iv).}
\end{remark}

\subsection{Interpolation operators}

In the above sections, we deal with polynomials of the best approximation and Fourier means. It turns out that Theorem~\ref{le1} can be also applied for interpolation operators. As an example, let us consider an interpolation analogue of the de la Vall\'ee Poussin means:
\begin{equation*}\label{V1}
 V_nf(t)=\frac1{3n}\sum_{k=0}^{6n-1}f\(t_k\)K_n\(t-t_k\),\quad t_k=\frac{\pi k}{3n},\quad t\in \T,
\end{equation*}
where
$$
K_n(t)=\frac12+\sum_{k=1}^{2n}\cos kt+\sum_{k=2n+1}^{4n-1}\frac{4n-k}{2n}\cos kt.
$$

Recall some basic properties of $V_n f$ (see~\cite{Sz}).

\begin{proposition}\label{lem1v}
The following assertions hold:

  {\rm (1)}\quad $\deg V_n f\le 4n-1$;

  {\rm (2)}\quad $V_nf\(t_k\)=f\(t_k\),\quad k=0,\dots,6n-1$;

  {\rm (3)}\quad $V_nT(t)=T(t)$ for any $T\in \mathcal{T}_{2n}$;

  {\rm (4)}\quad for all $f\in C(\T)$ and $r,n\in \N$, we have
$$
\Vert f-V_n f\Vert_{L_\infty(\T)} \lesssim \w_r(f,1/n)_\infty.
$$
\end{proposition}

Thus, noting that $V_n(V_{2n} f)=V_n f$ and using Theorem~\ref{le1}, Proposition~\ref{lem1v}, and the Nikolskii-Stechkin-type inequality~\eqref{eq+++}, we derive the following result.

\begin{theorem}
Let $f\in C(\T)$ and $r,n\in \N$. Then
$$
2^{-nr} \Vert (V_{2^n} f)^{(r)}\Vert_{L_\infty(\T)} \lesssim \w_r(f,2^{-n})_\infty \lesssim \sum_{k=n+1}^\infty 2^{-kr} \Vert (V_{2^k} f)^{(r)}\Vert_{L_\infty(\T)}.
$$
\end{theorem}




\section{Smoothness of approximation processes on $\R^d$}\label{sec6}

\subsection{Smoothness of best approximants.}
%

In what follows, the class of band-limited functions  $\mathcal{B}_p^\s$, $1\le p\le\infty$, $\s>0$, is given by
%
$$
\mathcal{B}_p^\s=\left\{\varphi  \in L_p(\R^d)\,:\,\supp\;\wh\varphi  (x)\subset \{x:\vert
x\vert  < \s\}\right\},
$$
where
$$
\wh g(x) = \int_{\R^d} g(y) e^{-i(x,y)} dy.
$$

Let
$$
E_{\s}(f)_{L_p(\R^d)} = \inf \big\{\Vert  f-\varphi  \Vert  _{L_p(\R^d)}: \varphi  \in \mathcal{B}_p^\s\big\}
$$
be the best approximation of $f$ and $P_{\s}(f)\in \mathcal{B}_p^\s$ be a best approximant of $f$ in $L_p(\R^d)$, that is,
$$
\Vert f-P_{\s}(f)\Vert_{L_p(\R^d)}=E_{\s}  (f)_{L_p(\R^d)}.
$$



We will use the following Jackson and Nikolskii-Stechkin inequalities, see, e.g.,~\cite[5.3.2]{Ti} and \cite[Theorem~3]{Wil} for the case $1\le p\le \infty$ and~\cite{KT20} for the case $0<p<1$:


\begin{equation}\label{JacksonSO_R}
E_\s(f)_p\lesssim \w_r\(f,\frac1\s\)_p,\quad f\in L_p(\R^d),\quad \s>0,\quad r\in \N,
\end{equation}

\begin{equation}\label{eq+++R}
 \Vert P_\s\Vert_{\dot W_p^r(\R^d)}\asymp \d^{-r}\w_r (P_n,\d)_{L_p(\R^d)},\quad P_\s\in \mathcal{B}_p^\s,\quad \s>0,\quad 0<\d\le\pi/ \s.
\end{equation}

%

Then, Theorem~\ref{le1} together with inequalities~\eqref{JacksonSO_R} and \eqref{eq+++R} imply the following result.
\begin{theorem}\label{PrGR}
  Let $f\in L_p(\R^d)$, $0< p\le \infty$, and $r\in \N$. Then
  \begin{equation*}\label{eq7R}
    2^{-nr}\Vert P_{2^n}(f)\Vert_{\dot W_p^r(\R^d)}\lesssim \omega_r(f,2^{-n})_{L_p(\R^d)}\lesssim \sum_{k=n+1}^\infty 2^{-kr}\Vert P_{2^k}(f)\Vert_{\dot W_p^r(\R^d)}.
  \end{equation*}
\end{theorem}

To sharpen this result in the case $1<p<\infty$, we will use Theorem~\ref{corthAB} with $G_n=\mathcal{B}_p^n$, $X=L_p(\R^d)$, and $Y=H_p^{\a}(\R^d)$, $\a>0$, where
$$
H_p^{\a}(\R^d)=\{g\in L_p(\R^d)\,:\, \Vert g \Vert_{\dot H_p^{\a}(\R^d)}=\Vert (-\Delta)^{\a/2} g\Vert_{L_p(\R^d)}<\infty\}
$$
is the fractional Sobolev spaces.
%
The corresponding $K$-functional and its realization are defined similarly to~\eqref{KTD} and~\eqref{RTD} and, moreover,
for any $f\in L_p(\R^d)$, $1<p<\infty$, and $\a>0$,
$$
K(f,t^{\a};L_p(\R^d),H_p^\a(\R^d))\asymp R(f,t^{\a};L_p(\R^d),\mathcal{B}_p^{1/t})\asymp \omega_\a(f,t)_{L_p(\R^d)},
$$
see~\cite{Wil}. This, in particular, implies
$$
\Vert  (-\Delta  )^{\a/2 } P_\s\Vert_{L_p(\R^d)}\lesssim n^\a \Vert  P_\s\Vert_{L_p(\R^d)},\quad P_\s\in\mathcal{B}_p^n,\quad 1<p<\infty.
$$


Thus, by Theorem~\ref{corthAB},  we obtain

\begin{theorem}\label{propAB}
  Let $f\in L_p(\R^d)$, $1<p<\infty$,  $\a>0$, $\tau=\max(2,p)$, and $\t=\min(2,p)$. Then
  \begin{equation*}
\begin{split}
    \(\sum_{k=n+1}^\infty 2^{-k\a\tau}\Vert (-\Delta)^{\a/2}P_{2^k}(f)\Vert_{L_p(\R^d)}^\tau\)^\frac1\tau&\lesssim \omega_\a(f,2^{-n})_{L_p(\R^d)}\\
&\lesssim  \(\sum_{k=n+1}^\infty 2^{-k\a\t}\Vert (-\Delta)^{\a/2}P_{2^k}(f)\Vert_{L_p(\R^d)}^\t\)^\frac1\t.
 \end{split}
\end{equation*}
\end{theorem}

\bigskip

\subsection{The case of Fourier multipliers operators.}

The  Mikhlin-H\"ormander multiplier theorem (cf. Assumption~\ref{H}) states that the condition
$$
\Big\vert  \frac{\partial^\b}{\partial^{\b_1} x_1\dots\partial^{\b_d}  x_d}\, \mu
(x)\Big\vert  \le A\vert  x\vert  ^{-\vert  \b  \vert  }, \q
\vert  \b  \vert  \equiv \b_1 + \dots + \b_d <\big[\frac
d2\big] + 1
$$
(see~\cite[p.~366]{Gr}) implies
$$
\Vert  T_\mu  f\Vert  _{L_p(\R^d)} \le C(A,p)\Vert  f\Vert
_{L_p(\R^d)},
$$
where
$
(T_\mu  f)^\wedge (x) = \mu  (x)\wh f(x).
$

Setting
$$
(\eta_\s f)^\wedge(x) = \eta\Big(\frac{\vert x\vert}{\s}\Big)\wh f(x)
$$
and
$$
\theta_0 (f) = \eta_1f \q\text{and} \q \theta_j (f)= \eta_{2^j}f
- \eta_{2^{j-1}} f \q\text{for} \q j\ge 1,
$$
we have the following analogue of the Littlewood-Paley theorem in the case $\mathcal{D}=\R^d$  (see~\cite[p.~20]{GrII} and~\cite[Theorem 4.1]{DDT}):
for $f\in L_p(\T^d)$, $1<p<\infty,$ and $\g>0$,
$$
\Big\Vert  \Big\{\sum^\infty  _{j=0} (\theta_j(f))^2\Big\}^{1/2}\Big\Vert_{L_p(\R^d)} \asymp \Vert  f\Vert_{L_p(\R^d)}
$$
and
$$
\Big\Vert  \Big\{\sum^\infty_{j=1} 2^{2j\a} \big(\theta_j
(f)\big)^2\Big\}^{1/2} \Big\Vert_{L_p(\R^d)} \asymp\Vert
(-\Delta)^{\a/2}  f\Vert_{L_p(\R^d)}.
$$

We introduce the operators $\Psi_\s$ and $\widetilde{\Psi}_\s$ as follows:
$$
(\Psi_\s f)^\wedge(x) = \psi\Big(\frac{\vert x\vert}{\s}\Big)\wh f(x),
$$
$$
(\widetilde{\Psi}_\s f)^\wedge(x) = \widetilde{\psi}\Big(\frac{\vert x\vert}{\s}\Big)\wh f(x),\quad \widetilde{\psi}(\xi)=\frac{\eta(|\xi|)}{\psi(2^{-m}\xi)},
$$
where a function $\psi:\R^d\to \C$ is such that $\supp \psi\subset [-1,1]^d$ and for some $m\in\Z_+$, $\psi(x)\neq 0$ for all $x\in [-2^{-m},2^{-m}]^d$.



We are now in a position to give a version of Theorem~\ref{thPsi}  in the case $\mathcal{D}=\R^d$.

\begin{theorem}\label{propVPRd}
  Let $f\in L_p(\R^d)$, $1<p< \infty$, $\a>0$, $\tau=\max(2,p)$, and $\t=\min(2,p)$.

  {\rm (A)} If  $\{\Psi_{2^k}\}$ are uniformly bounded in $L_p(\R^d)$, then
  \begin{equation*}\label{eq7-Rd}
    \(\sum_{k=n+1}^\infty 2^{-k\a\tau}\Vert (-\Delta)^{\a/2}\Psi_{2^k}f\Vert_{L_p(\R^d)}^\tau\)^\frac1\tau\lesssim \omega_\a(f,2^{-n})_{L_p(\R^d)}.
  \end{equation*}

  {\rm (B)} If  $\{\widetilde{\Psi}_{2^k}\}$ are uniformly bounded in $L_p(\R^d)$, then
    \begin{equation*}\label{eq7+Rd} \omega_\a(f,2^{-n})_{L_p(\R^d)}\lesssim  \(\sum_{k=n+1}^\infty 2^{-k\a\t}\Vert (-\Delta)^{\a/2}\Psi_{2^k}f\Vert_{L_p(\R^d)}^\t\)^\frac1\t.
  \end{equation*}
\end{theorem}

An analogue of Corollary~\ref{o-sn1} on $\R^d$, namely, inequality~\eqref{rr++++} holds for the following Fourier means:


1) the $\ell_q$-Fourier means given by
$$
\widehat{S_{n,q}} f(\xi)=\chi_{\{\xi\in\R^d\,:\,\Vert \xi\Vert_{\ell_q}\le n\}}(\xi) \widehat{f}(\xi) ,\quad q=1,\infty;
$$

2) the de~la~Vall\'ee Poussin-type means $\eta_{n}f(x)$;

3) the Riesz spherical  means $R_n^{\b,\d}$ given by
$$
\widehat{R_n^{\b,\d}} f(\xi)=\bigg(1- \bigg(\frac{\left|\xi\right|}n\bigg)^\b\bigg)_+^\delta \widehat{f}(\xi)
$$
for $\b>0$ and $\d>(d-1)/2$.



At the same time, an analogue of Corollary~\ref{fourier} on $\R^d$ is valid only for the de~la~Vall\'ee Poussin-type means and the Riesz spherical  means. Namely, for any $f \in L_p(\R^d)$, $p =1, \infty$, and $\a\in \N$, we have
\begin{equation*}
 2^{- n \a }
\Vert T_{2^n}f \Vert_{\dot W_p^\a(\R^d)} \lesssim \omega_{\a} \Big( f,
\frac{1}{2^n} \Big)_p\lesssim \sum\limits_{k =n+1}^{\infty} 2^{- k \a}
\Vert T_{2^k}f \Vert_{\dot W_p^\a(\R^d)},
\end{equation*}
where $T_{2^k} f=\eta_{2^k}f$  or $R_{2^k}^{\b,\d} f$ with $\d>(d-1)/2$.

Finally in this section, we give a characterization of the classical Besov spaces $B_{p,q}^s(\R^d)$ in terms of best approximants  and Fourier means. Using Theorems~\ref{PrGR}, \ref{propAB} and~\ref{propVPRd} and the same arguments as in Corollary~\ref{cor2.4}, we derive
\begin{corollary}
Let $1< p<\infty$, $0<q\le \infty$, and $0<s<\alpha$. We have 
\begin{equation}\label{BesovU}
  |f|_{B_{p,q}^s(\R^d)}\asymp \left( \sum_{k=1}^\infty 2^{(s-\a) q k} \Vert (-\Delta)^{\a/2}P_{2^k} (f)\Vert_{L_p(\R^d)}^q\right)^{\frac1q}, 
\end{equation}
where $P_{2^k}(f)$ stands for the best approximants or the Fourier means $\Psi_{2^k} f$ with the properties given in Theorem~\ref{propVPRd}.

In the case $p=1$ or $\infty$ and $\alpha\in \N$, $s<\alpha$, we have
\begin{equation*}\label{BesovU}
  |f|_{B_{p,q}^s(\R^d)}\asymp \left( \sum_{k=1}^\infty 2^{(s-\alpha) q k} \Vert P_{2^k}(f)\Vert_{\dot W_p^\alpha(\R^d)}^q\right)^{\frac1q}, 
\end{equation*}
where $P_{2^k}(f)$ stands for the best approximants, the de~la~Vall\'ee Poussin-type means $\eta_{n}f(x)$, or the Riesz spherical  means $R_n^{\b,\d}$ with $\d>(d-1)/2$.
\end{corollary}

Note that a similar assertion for the Gauss-Weierstrass semi-group
$W_t f(x)=(4\pi t)^{d/2}\int_{\R^d} e^{-\frac{|x-y|^2}{4t}}f(y)dy=(e^{-t|\xi|^2}\widehat{f}(\xi))(x)$, $t>0$,  was obtained in~\cite[Theorem~3.4.6, p. 198]{BB} and \cite[Section 1.13.2, pp. 76–81]{Tr}.



\section{Smoothness of approximation processes on $[-1,1]$}\label{sec7}

\subsection{Sharp inequalities for algebraic polynomials}

Let $L_{w,p}=L_p([-1,1]; w)$, $0< p\le\infty$,
be the space of all functions $f$ with the finite (quasi-)norm
\begin{equation*}
  \Vert f\Vert_{w,p}=\Vert f\Vert_{L_p([-1,1]; w)}=\(\int_{-1}^1 |f(x)|^p w(x){d}x\)^\frac1p,
\end{equation*}
where
\begin{equation*}
  w(x)=w_{a,b}(x)=(1-x)^a (1+x)^b,\quad a,b>-1,
\end{equation*}
is the Jacobi weight on $[-1,1]$.
In the unweighted case, $w(x)\equiv 1$, we  write $L_p=L_{p}[-1,1]$, $\Vert f\Vert_p=\Vert f\Vert_{L_p[-1,1]}$.


Further, let $\mathcal{P}_n$ be the set of all algebraic polynomials of degree at most $n$.
As usual, the error of the best approximation of a function $f\in L_{w,p}$ by algebraic polynomials is defined as follows:
\begin{equation*}
E_n(f)_{w,p}=\inf_{P\in \mathcal{P}_{n}}\Vert f-P\Vert_{w,p}.
\end{equation*}


Let $f\in L_p[-1,1]$, $0<p<\infty$, $r\in\N$, $\vp(x)=\sqrt{1-x^2}$, and $\s\ge 0$. Recall that the Ditzian-Totik modulus of smoothness $\w_r^\vp(f,\d)_{p}$ is given by
$$
\w_r^\vp (f,\d)_{p}=\sup_{|h|\le \d}\Vert \Dl_{h\vp}^r f\Vert_{L_p[-1,1]},
$$
    where
$$
\Dl_{h\vp(x)}^r f(x)=\left\{
                   \begin{array}{ll}
                     \displaystyle \sum_{k=0}^r (-1)^k\binom{r}{k}f\(x+\(\frac r2-k\)h\vp(x)\), & \hbox{$x\pm \frac r2 h\vp(x)\in [-1,1]$,} \\
                     \displaystyle 0, & \hbox{otherwise.}
                   \end{array}
                 \right.
$$

The Jackson-type theorem for the Ditzian-Totik moduli of smoothness is given by
\begin{equation*}\label{eqJALgeb+}
  E_n(f)_p\le C_{r,p}\w_r^\vp\(f,n^{-1}\)_p, \quad f\in L_p[-1,1],\quad 0< p<\infty,\quad n>r,
\end{equation*}
(see~\cite[Theorem~1.1]{DLY} for the case $0<p<1$ and~\cite[p.~79, Theorem~7.2.1]{book} for the case $p\ge 1$).
It is also well know, see, e.g.,~\cite{DHI}, that
$\w_r^\vp(f,\d)_p\le C_{r,p}\Vert f\Vert_p$ and
$
\w_r^{\vp}(f,\,2t)_p\leq C_{r,p}\w_r^{\vp}(f,t)_p.
$
Thus, taking into account the  following Nikolskii-Stechkin type inequality (see~\cite{DHI} and \cite{HuLi05})
$$
\w_r^\vp(P_n,\d)_p \asymp \d^r \Vert \vp^r
P_n^{(r)}\Vert_p,\quad  0<p<\infty,\quad P_n\in \mathcal{P}_n,\quad 0<\d\le n^{-1},
$$
we see that Theorem~\ref{le1} implies the following result (see also~\cite{HuLi05}).


\begin{theorem}\label{le1Pol}
  For any $f\in L_p[-1,1]$, $0<p\le \infty$, and $n>r$, we have
  \begin{equation*}
    2^{-r n}\Vert \vp^r P_{2^n}^{(r)}\Vert_{L_p[-1,1]}\lesssim \w_r^\vp(f,2^{-n})_{L_p[-1,1]}\lesssim \(\sum_{k=n+1}^\infty  2^{-r \l k}\Vert \vp^r P_{2^k}^{(r)}\Vert_{L_p[-1,1]}^\l\)^\frac1\l,
  \end{equation*}
where $\l=\min(1,p)$ and $P_n$ is a polynomial of the best approximation  of $f$ in $L_p[-1,1]$.
\end{theorem}

Now, we  are going to apply Theorems~\ref{corthAB} and~\ref{thLP} in the case of the weighted $L_p$ spaces for $1<p<\infty$.
First, we introduce  some notations.


For $a,b>-1$, denote by $P_k^\ab(x)$, $k\in \Z_+$, the system of Jacobi polynomials, orthogonal on $[-1,1]$, such that
$P_k^{(a,b)}(1)=\binom{k+a}{k}$,  $k\in \Z_+$.
Let also $R_k^{(a,b)}$ be the normalized Jacobi polynomials,
$
R_k^{(a,b)}(x)={P_k^{(a,b)}(x)}/{P_k^{(a,b)}(1)}$, $k\in \Z_+$.

The Fourier-Jacobi series of $f\in {L_{w,p}}$, $1\le p\le\infty$, $a,b>-1$, is given by
\begin{equation*}\label{eqFJs}
  f(x)\sim\sum_{k=0}^\infty c_k^\ab(f)\mu_k^\ab R_k^\ab(x),
\end{equation*}
with the Fourier coefficients
\begin{equation*}
c_k^\ab(f)=\int_{-1}^1 f(x)R_k^\ab(x)w(x){d}x,\quad k\in \Z_+,
\end{equation*}
and
$
\mu_k^\ab=\Vert R_k^\ab \Vert_{L_{w,2}}^{-2}\asymp k^{2a+1}.
$

Note that the Jacobi polynomials are the eigenfunctions of the differential operator
\begin{equation*}
Q(D) = Q_{\alpha  ,\beta  }(D)=\frac{-1}{w(x)}\frac{\mathrm{d}}{\mathrm{d}x}w(x)(1-x^2)\frac{\mathrm{d}}{\mathrm{d}x}\,,
\end{equation*}
\begin{equation*}
Q(D) P_k^\ab = \l_k^\ab P_k^\ab,\quad \l_k^\ab=k(k+a+b+1).
\end{equation*}
Then the corresponding $K$-functional  is given by~\eqref{Kfunc} with $\s=2$ and $\mathcal{D}=[-1,1]$.

Recall that by~\eqref{f3}  and \cite[Section 6]{DaDi}, we have
\begin{equation*}\label{eqValKfunc}
 K_\gamma  \big(f,Q(D),n^{-2\gamma}\big)_{L_{p,w}[-1,1]} \asymp \Vert f-\eta_n f\Vert_{L_{p,w}[-1,1]} +n^{-2\g}\Vert Q(D)^\gamma \eta_n f\Vert_{L_{p,w}[-1,1]} ,
\end{equation*}
where the de la Vall\'ee Poussin means $\eta_n f$ are given by
$$
\eta_n f(x)=\sum_{k=0}^\infty \eta\(\frac kn\) c_k^\ab(f)\mu_k^\ab R_k^\ab(x).
$$

Thus, employing Theorem~\ref{corthAB}, Theorem~\ref{thLP},  and the needed facts from~\cite[Section 6]{DaDi}, we obtain  the following result.

\begin{theorem}\label{propAB_pol}
  Let $f\in L_{p,w}[-1,1]$, $1<p<\infty$,  $\gamma>0$, $\tau=\max(2,p)$, and $\t=\min(2,p)$. Then
  \begin{equation}\label{eq7-P}
   \(\sum_{k=n+1}^\infty 2^{-2\gamma\tau k}\Vert Q(D)^\g\eta_{2^k}f\Vert_{L_{p,w}[-1,1]}^\tau\)^\frac1\tau\lesssim K_\g(f,Q(D),2^{-2n\g})_{L_{p,w}[-1,1]},
  \end{equation}
    \begin{equation}\label{eq7+P}
    K_\g(f,Q(D),2^{- 2 n\g })_{L_{p,w}[-1,1]}\lesssim \(\sum_{k=n+1}^\infty 2^{-2\gamma\t k}\Vert Q(D)^\g\eta_{2^k}f\Vert_{L_{p,w}[-1,1]}^\t\)^\frac1\t.
  \end{equation}

 Inequalities~\eqref{eq7-P} and~\eqref{eq7+P} are also valid if we replace the de la Vall\'ee Poussin means $\eta_{2^k} f$ by the best approximants $P_{2^k}(f)$, or by the Fourier-Jacobi means $\Psi_{2^k} f$ with the properties similar to those indicated in Theorem~\ref{thPsi}.
\end{theorem}

\begin{remark}
\emph{Note that the results given in  Theorems~\ref{le1Pol} and~\ref{propAB_pol} essentially improve the corresponding results for the best approximants  in $L_{p,w}[-1,1]$, $1\le p<\infty$, obtained early in~\cite[Theorem~8.3.1]{book},~\cite{bu1,Go, HuLi05, L}.}
\end{remark}

\subsection{Sharp inequalities for splines}

In this subsection, we consider approximation of
functions by splines in the space $L_p[0,1]$ with the (quasi-)norm $\Vert \cdot \Vert_p=\Vert \cdot \Vert_{L_p[0,1]}$.

Denote by $\mathcal{S}_{m,n}$ the set of all spline functions of
degree $m-1$ with the knots $t_j=t_{j,n}=j/n$, $j=0,\ldots,n,$ i.e.,
$S\in \mathcal{S}_{m,n}$ if $S\in C^{m-2}[0,1]$ and $S$ is some
algebraic polynomial of degree $m-1$ in each interval
$(t_{j-1},\,t_j)$, $j=1,\ldots,n$.

Let
$$
\E_{m,n}(f)_p=\inf_{S\in \Sp_{m,n}}\Vert f-S\Vert_{L_p[0,1]}
$$
be the best approximation of a function $f$ by
splines $S\in \Sp_{m,n}$ in $L_p[0,1]$.

The Jackson type inequality is given by (\cite[Theorem~1]{Os80}, see also~\cite[Ch. 12, p.~379]{DeLo})
\begin{equation}\label{JackSplineXXX}
  \E_{r,n}(f)_p\lesssim \omega_{r}(f,n^{-1})_p,
\end{equation}
where $f\in L_p[0,1]$, $0<p\le \infty$, $n\in \N$, and
$$
\w_r(f,\d)_p=\sup_{0<h\leq \d}\Vert \D_h^rf\Vert_{L_p[0,1-rh]}
$$
is the modulus of smoothness of order $r\in \N$.

Note that any spline $S_n\in \mathcal{S}_{r,n}$ can be represented (see \cite{Os80}) as follows:
\begin{equation*}\label{eq.lem.th2.4}
S_n(x)=P(x)+\sum_{j=1}^{n-1}a_j(x-t_j)_+^{r-1}\,,
\end{equation*}
where $P\in \mathcal{P}_{r-1}$, $x_+=x$ if $x\geq0$ and $x_+=0$ if $x<0$. Moreover, one has
\begin{equation}\label{eq.lem.th2.6}
C^{-1}n^{-(1+(r-1)p)}\sum\limits_{j=1}^{n-1}|a_j|^p \le \w_r(S_n,\,n^{-1})_p^p\le
C n^{-(1+(r-1)p)}\sum\limits_{j=1}^{n-1}|a_j|^p\,,
\end{equation}
where $C$ is a positive constant that depends only on $r$ and $p$.
Inequalities (\ref{eq.lem.th2.6}) were proved in~\cite[Lemma~2.1]{HuYu} (see also~\cite{Hu}) in the case $1\le p<\infty$. It is easy to see that the same also holds  in the case $0<p<1$.

It is important to mention that~\eqref{eq.lem.th2.6} implies that for any $S_n\in \mathcal{S}_{r,n}$, $n, r\in \N$, one has
\begin{equation}\label{eqVarp}
  \omega_{r}(S_n,n^{-1})_p\asymp n^{-(r-1)-\frac1p }V(S_n^{(r-1)})_p,\quad  0<p<\infty,
\end{equation}
where $V(f)_p$ denotes the  $p$-variation of the function $f$, that is,
\begin{equation*}
V(f)_p={\sup_{0=x_0<x_1<\dots<x_n=1}}\bigg(\sum_{k=0}^{n-1}
|f(x_{k+1})-f(x_{k})|^p\bigg)^{\frac 1p}.
\end{equation*}
In its turn,~\eqref{eqVarp} implies the following analogue of the Bernstein inequality
\begin{equation}\label{eqBer}
  n^{-(r-1)-\frac1p }V(S_n^{(r-1)})_p\lesssim \Vert S_n\Vert_p.
\end{equation}
Moreover, by~\eqref{JackSplineXXX} and~\eqref{eqVarp}, for any $S_n\in \mathcal{S}_{r,n}$, $n, r\in \N$, such that
$\Vert f-S_n\Vert_{L_p[0,1]}=\E_{r,n}(f)_p$, we have
\begin{equation}\label{realSp}
  \Vert f-S_n\Vert_p+n^{-(r-1+\frac1p)}V(S_n^{(r-1)})_p\asymp \w_r(f,n^{-1})_p.
\end{equation}

The above results allow us to apply Theorem~\ref{le1} to obtain the following result.

\begin{theorem}\label{le1Spl}
  Let $f\in L_p[0,1]$, $0<p< \infty$, $r,n\in  \N$, and $\l=\min(1,p)$. Then
  \begin{equation*}
    2^{- n (r-1+\frac1p)}
V(S_{2^k}^{(r-1)})_p\lesssim \w_r(f,2^{-n})_{p}\lesssim \(\sum_{k=n+1}^\infty \(2^{- k (r-1+\frac1p)}
V(S_{2^k}^{(r-1)})_p\)^\l\)^\frac1\l,
  \end{equation*}
where $S_{2^k}\in \Sp_{r,{2^k}}$ is such that
$\Vert f-S_{2^k}\Vert_{L_p[0,1]}=\E_{r,{2^k}}(f)_p$.
\end{theorem}

In the case $1<p<\infty$, using~\eqref{eqVarp}, \eqref{eqBer}, and~\eqref{realSp} and Theorem~\ref{corthAB}, we arrive at the next statement.

\begin{theorem}\label{o-sn1Spl}
Let $f \in L_p[0,1], 1 < p < \infty,$ $r,n\in \mathbb{N}$, and $\tau = \max(2,p)$, $\theta = \min(2,p)$.
 Then
\begin{equation*}
 \( \sum\limits_{k =n+1}^{\infty} 2^{- k (r-1+\frac1p) \tau}
V(S_{2^k}^{(r-1)})_p^{\tau}
\)^{\frac{1}{\tau}} \lesssim \omega_{r} ( f,
{2^{-n}} )_p \lesssim  \( \sum\limits_{k =n+1}^{\infty} 2^{- k (r-1+\frac1p) \t}
V(S_{2^k}^{(r-1)})_p^{\t}
\)^{\frac{1}{\t}},
\end{equation*}
where $S_{2^k}\in \Sp_{r,{2^k}}$ is such that
$\Vert f-S_{2^k}\Vert_{L_p[0,1]}=\E_{r,{2^k}}(f)_p$.
\end{theorem}

\bigskip

\section{Nonlinear methods of approximation}\label{sec8}

\subsection{Nonlinear wavelet approximation}

We restrict ourselves to the case of compactly supported biorthogonal wavelets and follow the discussion  in~\cite[Section~7]{DeVore}.
Let $\vp$ and $\widetilde{\vp}$ be two refinable compactly supported functions and let $\psi$ and $\widetilde{\psi}$ be their corresponding wavelets. Suppose that  $\vp$ and $\widetilde{\vp}$ are in duality as follows
$$
\int_\R \vp(x-j)\widetilde{\vp}(x-k)dx=\delta_{jk},
$$
where $\delta_{jk}$ is the Kronecker delta. Then each function $f\in L_p(\R)$ has the following wavelet decomposition:
\begin{equation*}\label{wa1}
  f=\sum_{I\in D} c_{I,p}(f)\psi_{I,p},\quad c_{I,p}(f)=\langle f,\widetilde{\psi}_{I,p/(p-1)}\rangle,
\end{equation*}
see, e.g.,~\cite{CDF} and~\cite{Da}. In the above formula, $D$ is the set of all dyadic intervals in $\R$, $I$ denotes the dyadic cube $I=2^{-k}(j+[0,1])$ associated with $j,k\in \Z$ and
$$
\psi_{I,p}(x)=|I|^{-1/p}\psi(2^kx -j).
$$

Let $\Sigma_n^w$ denote the set of all functions
\begin{equation*}\label{wa2}
  S=\sum_{I\in \Lambda} a_I\psi_I,
\end{equation*}
where $\Lambda\subset D$ is a set of dyadic intervals of cardinality $\#\Lambda\le n$. Thus $\Sigma_n^w$ is
the set of all functions which are a linear combination of $n$ wavelet functions.
We define
$$
\s_{n}^w(f)_p=\inf_{S\in \Sigma_{n}^w} \Vert f-S\Vert_{L_p(\R)}.
$$

Let $B_{p,q}^r(\R)$, $r>0$, $0<p,q\le\infty$,
be the classical Besov spaces.
The Jackson and Bernstein  type inequalities are given in the following two propositions (see~\cite[Corollary~4.1 and Theorem~4.3]{CDH}).

\begin{proposition}\label{lewa1}
  Let $1<p<\infty$, $r>0$, and $f\in L_p(\R)$, $1/\g=r+1/p$. If $\psi$ has $m$ vanishing moments with $m>r$ and $\psi$ is in $B_{\g,q}^\rho (\R)$ for some $q>0$ and some $\rho>r$, then
\begin{equation*}\label{wa3}
  \s_n^w(f)_p\lesssim  K\(f,n^{-r}; L_p(\R),B_{\g,\g}^r(\R)\),\quad n\in\N.
\end{equation*}
\end{proposition}

\begin{proposition}\label{lewa2}
  Let $1<p<\infty$, $r>0$, $1/\g=r+1/p$.  If 
$
S=\sum_{I\in \Lambda} c_{I,p}(f) \psi_{I,p},
$
with $\# \Lambda \le n$, then 
\begin{equation*}\label{wa4}
  |S|_{B_{\g,\g}^r(\R)} \lesssim n^r \Vert S\Vert_{L_p(\R)}.
\end{equation*}
\end{proposition}


We will also use the fact that there exists ${Q}_n f\in \Sigma_n^w$ such that
$\Vert f-{Q}_n f\Vert_{L_p(\R)} \lesssim \s_n^w(f)_p $ and
$$
K\(f,n^{-r};L_p(\R),B_{\g,\g}^r(\R)\) \asymp \Vert f-{Q}_n f\Vert_{L_p(\R)}+n^{-r} \vert {Q}_n f\vert_{B_{\g,\g}^r(\R)}
$$
(see for details~\cite{CDH}). This realization result in particular implies the Nikolskii-Stechkin-type inequality
$$
K\(S,n^{-r};L_p(\R),B_{\g,\g}^r(\R)\) \asymp n^{-r} \vert S\vert_{B_{\g,\g}^r(\R)},\quad S\in \Sigma_n^w.
$$
Thus, in light of  Theorem~\ref{corthAB}, Propositions~\ref{lewa1} and~\ref{lewa2}, we obtain the following result.

\begin{theorem}\label{thNsp}
Under conditions of Proposition~\ref{lewa1}, we have
\begin{equation*}\label{rr}
\begin{split}
\( \sum\limits_{k =n+1}^{\infty}
2^{-r\tau k} \vert {P}_{2^k} f\vert_{B_{\g,\g}^r(\R)}^{\tau}
\)^{\frac{1}{\tau}}
&\lesssim
K\(f,2^{-r n}; L_p(\R),B_{\g,\g}^r(\R)\)\\
&\lesssim  \( \sum\limits_{k =n+1}^{\infty}
2^{-r\t k} \vert {P}_{2^k} f\vert_{B_{\g,\g}^r(\R)}^{\t}
\)^{\frac{1}{\t}},
\end{split}
\end{equation*}
where $P_{2^k}f\in \Sigma_{2^k}^w$ is such that $\Vert f-P_{2^k} f\Vert_{L_p(\R)}=\s_{{2^k}}^w(f)_p$ and $\tau=\max(2,p)$, $\t=\min(2,p)$.
\end{theorem}

%

As a corollary, we obtain the characterization of the Besov space $B_{X,q}^r$ (interpolation space) given in~\eqref{Besov} with $X=L_p(\R)$ and $\Omega(f,2^{-k})_X=K(f,2^{-r k},L_p(\R),B_{\g,\g}^r(\R))$.

\begin{corollary}
Under conditions of Proposition~\ref{lewa1}, if $0<\s<r$ and $0<q\le \infty$, then
$$
|f|_{B_{X,q}^\s(\R)}\asymp \left( \sum_{k=1}^\infty 2^{(\s-r) q k} \vert {P}_{2^k} f\vert_{B_{\g,\g}^r(\R)}^q\right)^{\frac1q}, 
$$
where $P_{2^k} f\in \Sigma_{2^k}^w$ is such that
$\Vert f-P_{2^k} f\Vert_{L_p(\R)}=\s_{2^k}^w(f)_p$.
\end{corollary}

\subsection{Free knot piecewise polynomial approximation}

Let $r\in \N$ be fixed and for each $n=1,2,\dots$, let $\Sigma_{r,n}$ be the space of piecewise polynomials of degree  $r$ with $n$ pieces on $[0,1]$. That is, for each element $S\in \Sigma_{r,n}$ there is a partition $\Lambda$ of $[0,1]$ consisting of $n$ disjoint intervals $I\subset [0,1]$ and polynomials $P_I\in \mathcal{P}_r$ such that
$$
S=\sum_{I\in \Lambda} P_I \chi_I.
$$
For each $0<p<\infty$, we define the error of the best approximation by
$$
\s_{r,n}(f)_p=\inf_{S\in \Sigma_{r,n}} \Vert f-S\Vert_{L_p[0,1]}.
$$

%

Recall the well-known Jackson-type inequality (see~\cite[Theorem~2.3]{Pe88}).

\begin{proposition}\label{Jacknelsp}
Let $f \in L_p[0,1]$, $0<p<\infty$, $r>0$, $k\in \N$, and $1/\g=r+1/p$. Then
\begin{equation}\label{NJackS}
  \s_{r,n}(f)_p\lesssim K\(f,n^{-r};L_p[0,1],B_{\g,\g; k}^r[0,1]\),\quad n\in \N,
\end{equation}
where
$B_{\g,\g; k}^r[0,1]$ is the non-periodic Besov space, which consists of $f\in L_\g[0,1]$ such that
$$
|f|_{B_{p,q;\, k}^r[0,1]}=\(\int_0^{1/k} \(t^{-r} \omega_k(f,t)_{L_\g[0,1]}\)^\g \frac{dt}{t}\)^{1/\g}<\infty.
$$
\end{proposition}

Now using~\eqref{NJackS} and 
Theorem~\ref{le1}, we derive  the following result.


\begin{theorem}\label{thNsp}
Under conditions of Proposition~\ref{Jacknelsp}, we have
\begin{equation*}\label{rr}
\begin{split}
   K\(S_{2^n},2^{-rn};L_p[0,1],B_{\g,\g; k}^r[0,1]\)
&\lesssim
K\(f,2^{-rn};L_p[0,1],B_{\g,\g; k}^r[0,1]\)
\\
&\lesssim  \( \sum\limits_{k =n+1}^{\infty}
K\(S_{2^k},2^{-rk};L_p[0,1],B_{\g,\g; k}^r[0,1]\)^{\l}
\)^{\frac{1}{\l}},
\end{split}
\end{equation*}
where $S_{2^k}\in \Sigma_{r,2^k}$ is such that
$\Vert f-S_{2^k}\Vert_{L_p[0,1]}=\s_{r,2^k}(f)_p$ and $\l=\min(p,1)$.
\end{theorem}


Finally, we characterize the Besov  space $B_{X,q}^r$  given in~\eqref{Besov} with
 $X=L_p[0,1]$ and $\Omega(f,2^{-k})_X=K\(f,2^{-r k},L_p[0,1],B_{\g,\g; k}^r[0,1]\)$.

\begin{corollary}
Let $0<\s<r$ and $0<q\le\infty$. We have
$$
|f|_{B_{X,q}^\s[0,1]}\asymp \left( \sum_{k=1}^\infty 2^{\s q k} K\(S_{2^k},2^{-rk};L_p[0,1],B_{\g,\g; k}^r[0,1]\)^q\right)^{\frac1q}, 
$$
where $S_{2^k}\in \Sigma_{2^k,r}$ is such that
$\Vert f-S_{2^k}\Vert_{L_p[0,1]}=\s_{r,{2^k}}(f)_p$.
\end{corollary}



\section{Optimality}\label{secOpt}\label{sec9}

In the previous sections, we  derived the following inequalities:
\begin{equation}\label{optimal}
    \(\sum_{k=n+1}^\infty 2^{-k\a\tau}\Vert P_{2^k}(f)\Vert_{Y}^\tau\)^\frac1\tau
    \lesssim K(f,2^{-n\a};L_p,Y)
    \lesssim
    \(\sum_{k=n+1}^\infty 2^{-k\a\theta}\Vert P_{2^k}(f)\Vert_{Y}^\theta\)^\frac1\theta,
 \end{equation}
 where $f \in L_p$, $1\le p\le\infty$,  $$\tau=\left\{
                                            \begin{array}{ll}
                                              \max(p,2), & \hbox{$1<p<\infty$,} \\
                                              \infty, & \hbox{otherwise}
                                            \end{array}
                                          \right.,
\qquad
\theta=\left\{
                                            \begin{array}{ll}
                                              \min(p,2), & \hbox{$p<\infty$,} \\
                                              1, & \hbox{$p=\infty$,}
                                            \end{array}
                                          \right.
 $$
$Y$ is an appropriate smooth function space, and
$P_n(f)$ is a suitable approximation method.
In this section, we show that the parameters $\theta$ and $\tau$ are optimal.

For this, we restrict ourselves to the case of $\mathcal{D}=\T$ and approximation of periodic $L_p$-functions by $S_n(f)$, the $n$-th partial sums of the Fourier series of $f$,  and the de la Vall\'ee Poussin means $\eta_n f$.

Recall that if $f \in L_p(\T), 1 < p < \infty,$ then inequality (\ref{optimal}) in particular implies 
\begin{equation}\label{optimal-p}
 \( \sum\limits_{k =n+1}^{\infty} 2^{- k \a \tau}
\Vert S_{2^k}^{(\a)}(f) \Vert_p^{\tau}
\)^{\frac{1}{\tau}} \lesssim \omega_{\a} \Big( f,
\frac{1}{2^n} \Big)_p \lesssim \( \sum\limits_{k =n+1}^{\infty}
2^{- k \a \theta} \Vert S_{2^k}^{(\a)}(f) \Vert_p^{\theta}
\)^{\frac{1}{\theta}}.
\end{equation}
If $f \in L_p(\T),  p=1, \infty,$ and $P_n(f)=\eta_n f,$  estimate (\ref{optimal}) can be written by
\begin{equation*}\label{optimal+p}
2^{- \a n} \Vert (\eta_{2^n}f)^{(\a)}\Vert_{L_p(\T)}  \lesssim  \w_\a(f,2^{-n})_{L_p(\T)}\lesssim \sum_{k=n}^\infty 2^{-2 \a k} \Vert (\eta_{2^k}f)^{(\a)}\Vert_{L_p(\T)}.
  \end{equation*}

\subsection{Optimality of (\ref{optimal}) in the case $1<p<\infty$}

In this subsection, we deal with  not only sharpness of the parameters  $\tau = \max(2,p)$ and $\theta = \min(2,p)$ but we also show that for the classes
of functions with lacunary and general monotone Fourier coefficients,
inequality (\ref{optimal}) becomes an equivalence with $\tau = \theta = 2$ and $\tau = \theta = p$, respectively.


We start with lacunary series and first give a simple proof of Zygmund's theorem in $L_p$, $1<p<\infty$, based on the Littlewood--Paley technique given in  Section~\ref{SecReal}. We  deal with the general case of functions represented by
$$
f\sim \sum_{k=0}^\infty A_k f,\quad A_k f =\sum^{d_k}_{\ell=1} \, \langle f,\psi_{k,\ell} \rangle\,\psi_{k,\ell}.
$$
For convenience,  we  suppose that the dimension $d_k=1$ for all $k\in \Z_+$.

We will say that the Fourier expansion of $f\in L_{p,w}(\mathcal{D})$ is lacunary, written $f\in \Lambda$, if
$
f\sim \sum_{j=0}^\infty A_{2^j}f,
$
i.e., $A_{k}f=0$ for $k\neq 2^j$, $j\in \Z_+$.

Let us first derive an analogue of Zygmund's theorem.

\begin{lemma}\label{thZ1}
  Let $1<p<\infty$, $f\in \Lambda$, and Assumption~\ref{H} hold. Suppose that $w\in L_1(\mathcal{D})$ and the functions $\psi_k=\psi_{k,1}$ are such that
\begin{equation}\label{bounds}
  0<\xi_2\le \Vert \psi_k\Vert_{p,w}\le \xi_1<\infty\quad\text{for any}\quad k\in \Z_+.
\end{equation}
Then
  $$
  \Vert f\Vert_{p,w}\asymp \(\sum_{k=0}^\infty c_{2^k}(f)^2\)^\frac12,\quad c_k(f)=\int_{\mathcal{D}} f \psi_k\, w.
  $$
In particular, $\Vert f\Vert_{p,w}\asymp \Vert f\Vert_{2,w}$.
\end{lemma}

\begin{proof}
First, let us prove the estimate from above.
Let $1<p\le 2$. Then by H\"older's inequality and Parseval's inequality, we obtain
$
\Vert f\Vert_{p,w}\lesssim \Vert f\Vert_{2,w}\asymp \(\sum_{k=0}^\infty c_{2^k}(f)^2\)^\frac 12.
$
If $p\ge 2$, noting that
\begin{equation*}
  \begin{split}
     \theta_j(A_{2^k}f)=(\eta_{2^j}-\eta_{2^{j-1}})(A_{2^k}f)=\left\{
                                                                  \begin{array}{ll}
                                                                    A_{2^{j-1}} f, & \hbox{$j=k+1$,} \\
                                                                    0, & \hbox{$j\neq k+1$,}
                                                                  \end{array}
                                                                \right.
   \end{split}
\end{equation*}
and using the Littlewood–Paley decomposition  (Theorem~\ref{J3.0}), Minkowski's inequality, and~\eqref{bounds}, we derive
\begin{equation*}
  \begin{split}
     \Vert f\Vert_{p,w}&\asymp \bigg\Vert \bigg(\sum_{k=0}^\infty \theta_k (f)^2\bigg)^\frac12\bigg\Vert_{p,w}
     = \bigg\Vert \bigg(\sum_{k=0}^\infty A_{2^k} (f)^2\bigg)^\frac12\bigg\Vert_{p,w}\\
&= \(\int_{\mathcal{D}} \(\sum_{k=0}^\infty (c_{2^k}(f)\psi_{2^k})^2\)^\frac p2 w\)^\frac1p\\
&\le \(\sum_{k=0}^\infty \( \int_{\mathcal{D}} |c_{2^k}(f)\psi_{2^k}|^p w \)^\frac{2}{p}\)^\frac12\\
&\le \(\sum_{k=0}^\infty |c_{2^k}(f)|^2\)^\frac12\max_{k}\(\int_{\mathcal{D}}|\psi_{2^k}|^p w\)^\frac1p
\lesssim \(\sum_{k=0}^\infty |c_{2^k}(f)|^2\)^\frac12.
   \end{split}
\end{equation*}

To show the inverse inequality for $p\le 2$, we similarly obtain 
\begin{equation*}
  \begin{split}
      \Vert f \Vert_{p,w} &\gtrsim \(\int_{\mathcal{D}} \(\sum_{k=0}^\infty ( c_{2^k}(f)\psi_{2^k})^2\)^\frac{p}{2}w\)^\frac1p\\
      &\ge \(\sum_{k=0}^\infty \(\int_{\mathcal{D}} |c_{2^k}(f)\psi_{2^k}|^p w\)^\frac2p\)^\frac12\\
      &\ge \(\sum_{k=0}^\infty c_{2^k}(f)^2\)^\frac12\min_{k}\Vert \psi_{2^k}\Vert_{p,w}\gtrsim\(\sum_{k=0}^\infty c_{2^k}(f)^2\)^\frac12.
   \end{split}
\end{equation*}
If $p\ge 2$, H\"older's inequality implies
$
\Vert f\Vert_{2,w}\lesssim \Vert f\Vert_{p,w},
$
which proves the lemma.
\end{proof}

\begin{remark}
\emph{As an example of the system $\{\psi_k\}$ in Lemma~\ref{thZ1}, one can take the trigonometric system, the Walsh system, systems of the Chebyshev polynomials and, more generally, the system of normalized Jacobi polynomials for specific range of parameters $\a,\b>-1$ indicated in~\cite{ABD}.}
\end{remark}

\begin{theorem}\label{thOpt2}
Under all assumptions of Lemma~\ref{thZ1}, we have for $f\in L_{p,w}(\mathcal{D})\cap \Lambda$
\begin{equation*}
  \(\sum_{k=n+1}^\infty 2^{- 2\g \s k} \Vert Q(D)^\gamma \eta_{2^k} f\Vert_{p,w}^2\)^\frac12\asymp K_\gamma(f,Q(D),2^{-n \g \s})_{p,w},\quad \g>0.
\end{equation*}
\end{theorem}

\begin{proof}
Using the realization result~\eqref{f3} and Lemma~\ref{thZ1}, we get
\begin{equation}\label{Opt1}
  \begin{split}
      K_\gamma(f,Q(D),2^{-n \g \s})_{p,w}&\asymp \Vert f-\eta_{2^n}f\Vert_{p,w}+2^{-\g\s n}\Vert Q(D)^\g\eta_{2^n} f\Vert_{p,w}\\
      &\asymp \(\sum_{k=n-1}^\infty c_{2^k}(f)^2\)^\frac12+2^{-\g\s n}\(\sum_{k=1}^{n-1} 2^{2 \g\s k}c_{2^k}(f)^2\)^\frac12
   \end{split}
\end{equation}
and
\begin{equation*}
  2^{-2 \g\s k}\Vert Q(D)^\g\eta_{2^k} f\Vert_{p,w}^2 \asymp 2^{-2 \g \s k}\sum_{l=1}^{k-1} 2^{2 \g \s l}c_{2^l}(f)^2.
\end{equation*}
Then 
\begin{equation*}
  \begin{split}
      \sum_{k=n+1}^\infty 2^{- 2\g \s k} \Vert Q(D)^\gamma \eta_{2^k} f\Vert_{p,w}^2&\asymp \sum_{k=n+1}^\infty 2^{- 2\g \s k} \sum_{l=1}^{k-1} 2^{2 \g \s l}c_{2^l}(f)^2\\
      &=\sum_{k=n+1}^\infty 2^{- 2\g \s k} \(\sum_{l=1}^{n} \,+\, \sum_{l=n+1}^{k-1}  2^{2 \g \s l}c_{2^l}(f)^2\)\\
      &\asymp   2^{- 2\g \s n} \sum_{l=1}^{n} 2^{2 \g \s l}c_{2^l}(f)^2+ \sum_{l=n+1}^{\infty}  2^{2 \g \s l}c_{2^l}(f)^2\\
      &\asymp K_\gamma(f,Q(D),2^{-n \g \s})_{p,w}^2.
   \end{split}
\end{equation*}
\end{proof}
In particular, 
for  the classical Fourier series on $\mathcal{D}=\T$ we obtain 
  \begin{equation}\label{eqOptTrigL1--}
    \w_\a(f,2^{-n})_{L_p(\T)}\asymp \(\sum_{k=n}^\infty 2^{-2 \a k} \Vert (S_{2^k}f)^{(\a)}\Vert_{L_p(\T)}^2\)^\frac12,\quad f\in L_p(\T)\cap \Lambda,
  \end{equation}
  where $1<p<\infty$ and $\a>0$;
cf.~\eqref{optimal-p}.

  \begin{remark}\label{remOpt1}
\emph{It is clear that (\ref{eqOptTrigL1--})
gives the sharpness of the parameter  $\theta$ for $p\ge2$ and $\tau$ for $p\le 2$ in inequality (\ref{optimal-p}).}
  \end{remark}

\begin{proof}
Assume that  $p\ge2$ and there holds
  \begin{equation}\label{eqOptTrigL1---}
    \w_\a(f,2^{-n})_{L_p(\T)}
\lesssim
 \(\sum_{k=n}^\infty \left(2^{- \a k} \Vert (S_{2^k}f)^{(\a)}\Vert_{L_p(\T)} \right)^{2+\varepsilon}\)^\frac1{2+\varepsilon}
  \end{equation}
 with some $\varepsilon>0$.
Consider $f(x)=\sum_{n=1}^\infty a_{2^n} \cos 2^nx$, where $a_{2^n}=1/n$. Then $f\in L_p(\T)\cap \Lambda$ and, by~\eqref{Opt1}, one has
$$    \w_\a(f,2^{-n})_{L_p(\T)}\asymp
\(\sum_{k=n}^\infty a_{2^k}^2\)^\frac12+2^{-\a n}\(\sum_{k=1}^{n} 2^{2 \a k}a_{2^k}^2\)^\frac12
\asymp \frac1{n^{1/2}},
$$
$$
2^{-\a k} \Vert (S_{2^k}f)^{(\a)}\Vert_{L_p(\T)}
\asymp \frac1{k},\qquad
\(\sum_{k=n}^\infty 2^{-(2+\varepsilon) \a k} \Vert (S_{2^k}f)^{(\a)}\Vert_{L_p(\T)}^{2+\varepsilon}\)^\frac1{2+\varepsilon}\asymp
n^{-\frac{1+\varepsilon}{2+\varepsilon}},
$$
which contradicts  (\ref{eqOptTrigL1---}).
Similarly, if $p\le2$, then the inequality
  \begin{equation*}\label{eqOptTrigL1---+}
    \w_\a(f,2^{-n})_{L_p(\T)}
\gtrsim
 \(\sum_{k=n}^\infty \left(2^{- \a k} \Vert (S_{2^k}f)^{(\a)}\Vert_{L_p(\T)} \right)^{2-\varepsilon}\)^\frac1{2-\varepsilon}
  \end{equation*}
with some  $\varepsilon\in (0,2)$ does not hold for $f(x)=\sum_{n=1}^\infty a_{2^n} \cos 2^nx\in L_p$, where $a_{2^n}=n^{-1/(2-\varepsilon)}$.
\end{proof}


Now, let us consider the case of the classical Fourier series with general monotone coefficients. In what follows,
we say (see~\cite{tikhonov}) that a (complex) sequence $\{d_n\}$ is general monotone, written $\{d_n\}\in GM$, if
$$\sum_{k=n}^{2n}|d_k-d_{k+1}|\le C |d_n|,$$
where $C$ does not depend on $n$. Note that any monotone (quasi-monotone) sequences are general monotone.
We denote by $\widehat{GM}$ the  class of integrable  functions such that $f(x)\sim \sum_{n=1}^\infty (a_n\cos nx+b_n\sin nx)$ with
$\{a_n\}, \{b_n\}\in GM$.

\begin{theorem}\label{thOpt3}
  Let $f\in L_p(\T)\cap \widehat{GM}$, $1<p<\infty$, and $\a>0$. Then
  \begin{equation}\label{eqOptTrigL1-}
    \w_\a(f,2^{-n})_{L_p(\T)}\asymp \(\sum_{k=n}^\infty 2^{-p \a k} \Vert (S_{2^k}f)^{(\a)}\Vert_{L_p(\T)}^p\)^\frac1p.
  \end{equation}
\end{theorem}

\begin{proof}
First, we recall the following Hardy--Littlewood theorem:
$$
\Vert f\Vert_{L_p(\T)}  \asymp
\(\sum_{n=1}^\infty (|a_n|+|b_n|)^p n^{p-2}\)^\frac1p,\quad f\in L_p(\T)\cap \widehat{GM},\quad 1<p<\infty.
 $$
This is a well-known fact for functions with monotone coefficients, see~\cite[Ch.~XII]{Z}. For the class $\widehat{GM}$ (in fact for a more general class and for Lorentz spaces) this has been recently proved in~\cite{muk}.
Moreover, it is also shown in~\cite{muk} that
  \begin{equation*}\label{Opt3--}
    \begin{split}
   \w_\a(f,n^{-1})_{L_p(\T)}&\asymp
n^{- \a }
\(\sum_{k=0}^n (|a_{k}|+|b_{k}|)^p k^{p\a+p-2}\)^\frac1p
+
\(\sum_{k=n}^\infty (|a_{k}|+|b_{k}|)^p k^{p-2}\)^\frac1p.
\end{split}
\end{equation*}

Now we note that the sequences $\{d_1,\cdots,d_n,0,0,\cdots\}$ and $\{n^\alpha d_n\}$ belong to $GM$  whenever  $\{d_n\}\in GM$, which implies that the Hardy--Littlewood theorem can be applied for the partial Fourier sums of $f$.
Moreover, since any general monotone sequence $\{d_n\}$ satisfies the following property, see~\cite{tikhonov}:
$|d_k|\le C |d_n|$ for $n\le k\le 2n$,
we have
$$
\Vert (S_{2^n} f)^{(\a)}\Vert_{L_p(\T)}\asymp
\(\sum_{k=0}^n (|a_{2^k}|+|b_{2^k}|)^p 2^{k(p\a+p-1)}\)^\frac1p.
$$
Thus, we derive 
  \begin{equation*}
    \begin{split}
   \w_\a(f,2^{-n})_{L_p(\T)}
&\asymp
2^{- \a n}
\(\sum_{k=0}^n (|a_{2^k}|+|b_{2^k}|)^p 2^{k(p\a+p-1)}\)^\frac1p
+
\(\sum_{k=n}^\infty (|a_{2^k}|+|b_{2^k}|)^p 2^{k(p-1)}\)^\frac1p
\\
&\asymp
\(\sum_{k=n}^\infty 2^{-p \a k}
\sum_{l=0}^k (|a_{2^l}|+|b_{2^l}|)^p 2^{l(p\a+p-1)}
\)^\frac1p
\\
&\asymp
  \(\sum_{k=n}^\infty 2^{-p \a k} \Vert (S_{2^k}f)^{(\a)}\Vert_{L_p(\T)}^p\)^\frac1p,
\end{split}
\end{equation*}
completing the proof.
\end{proof}

  \begin{remark}
\emph{Similarly to Remark~\ref{remOpt1}, equivalence~\eqref{eqOptTrigL1-} provides the sharpness of the parameter  $\theta$ for $p\le2$ and $\tau$ for $p\ge 2$ in (\ref{optimal-p}).}
  \end{remark}

\subsection{Optimality  of the right-hand inequality in~\eqref{optimal} for $p=1$ and $p=\infty$}
We start by  obtaining two simple  results for lacunary Fourier series.
\begin{theorem}\label{thOptTrigL1}
  Let $f\in L_1(\T)\cap \Lambda$ and $\a>0$. Then
  \begin{equation*}\label{eqOptTrigL1}
    \w_\a(f,2^{-n})_{L_1(\T)}\asymp \(\sum_{k=n}^\infty 2^{-2 \a k} \Vert (\eta_{2^k}f)^{(\a)}\Vert_{L_1(\T)}^2\)^\frac12.
  \end{equation*}
\end{theorem}

\begin{proof}
The proof repeats the one of Theorem~\ref{thOpt2} since by Zygmund's theorem (see~\cite[Theorem 3.7.4]{Gr}), we have
  \begin{equation*}\label{Opt2}
      \w_\a(f,2^{-n})_{L_1(\T)}\asymp \(\sum_{k=n}^\infty |c_{2^k}|^2\)^\frac12+2^{-\a n}\(\sum_{k=1}^n 2^{2\a k}|c_{2^k}|^2\)^\frac12,
  \end{equation*}
where $\{c_k\}$ are the Fourier coefficients of $f$.
\end{proof}

\begin{theorem}\label{thOptTrigLinfty}
  Let $f\in L_\infty(\T)\cap \Lambda$ and $\a>0$. Then
  \begin{equation*}\label{eqOptTrigLinfty}
    \w_\a(f,2^{-n})_{L_\infty(\T)}\asymp\sum_{k=n}^\infty 2^{- \a k} \Vert \eta_{2^k}^{(\a)}f\Vert_{L_\infty(\T)}.
  \end{equation*}
\end{theorem}

\begin{proof}
  By Stechkin's theorem (see~\cite[Theorem 3.7.6]{Gr}), we have
  \begin{equation*}
    \begin{split}
       \sum_{k=n}^\infty 2^{- \a k} \Vert (\eta_{2^k}f)^{(\a)}\Vert_{L_\infty(\T)} &\asymp \sum_{k=n}^\infty 2^{-\a k}\sum_{s=1}^{n-1} 2^{\a s}|c_{2^s}|+
       \sum_{k=n}^\infty 2^{-\a k}\sum_{s=n}^k 2^{\a s} |c_{2^s}|\\
       &\asymp 2^{-\a n} \Vert (\eta_{2^n}f)^{(\a)}\Vert_{L_\infty(\T)} +\sum_{k=n}^\infty |c_{2^k}|\\
       &\asymp 2^{-\a n} \Vert (\eta_{2^n}f)^{(\a)}\Vert_{L_\infty(\T)} +E_{2^{n}}(f)_\infty \asymp \w_\a(f,2^{-n})_{L_\infty(\T)}.
    \end{split}
  \end{equation*}
\end{proof}

Note that Theorem~\ref{thOptTrigLinfty} shows that in the case $p=\infty$, the right-hand inequality~\eqref{optimal} is sharp for $\theta=1$, in other words this inequality cannot be improved for some $\theta>1$ in the general case. At the same time, we remark that Theorem~\ref{thOptTrigL1} only shows that in the case $p=1$,  the right-hand inequality~\eqref{optimal} is sharp for $\theta=2$, that is,~\eqref{optimal} cannot be sharpen with any $\t>2$.


Now we show that~\eqref{optimal} is in fact sharp for $\t=1$.

\begin{theorem}\label{Opt2LL2}
  Let $\a\in \N$. Then for any $q>1$ there exists a function $f\in L_1(\T)$ such that
  \begin{equation}\label{Opt2LL2-}
   \omega_\a(f,2^{-n})_{L_{1}(\T)}\le C  \(\sum_{k=n+1}^\infty 2^{-q \a k}\Vert (\eta_{2^k}f)^{(\a)}\Vert_{L_1(\T)}^q\)^\frac1q
  \end{equation}
   is not valid with a constant $C$ independent of $n$ and $f$.
\end{theorem}

\begin{proof}
We will use the following well-known Kolmogorov's estimates for the $L_1$-norms of trigonometric series:

\begin{equation}\label{eqSum+1}
  \int_0^\pi \left|\sum_{k=1}^\infty a_k \cos kx \right|dx\lesssim\sum_{k=1}^\infty k|\Delta^2 a_k|,
\end{equation}
\begin{equation}\label{eqSum+2}
  \int_0^\pi \left|\sum_{k=1}^\infty a_k \sin kx \right|dx\lesssim \sum_{k=1}^\infty k|\Delta^2 a_k|+\sum_{k=1}^\infty \frac{|a_k|}{k},
\end{equation}
where $\Delta^2 a_k =a_{k+2}-2a_{k+1}+a_k$. Inequality~\eqref{eqSum+1} was obtained in \cite{Kolm1}, see  also \cite{Te64}; for inequality~\eqref{eqSum+2} see
 \cite{Te64}.

We will also need the following estimate for the error of the best approximation
given by (see~\cite[Lemma 2]{Ge})
\begin{equation}\label{eqSumbestL1}
 E_n(g)_{L_1(\T)}\gtrsim \left|\sum_{k=n+1}^\infty \frac{a_k}{k}\right|,\quad g(x)\sim \sum_{k=1}^\infty a_k\sin kx\in L_1(\T).
\end{equation}

Now consider the function
$$
  f_N(x)=\sum_{k=1}^N\frac{\sin kx}{\log^\g (k+1)},
  $$
 where $N>2^n$ and $0<\g<1/q$. By the Jackson inequality and~\eqref{eqSumbestL1}, we obtain
\begin{equation}\label{eqE_nL1}
\begin{split}
  \w_\a(f_N,2^{-n})_{L_1(\T)} &\gtrsim E_{2^n}(f_N)_{L_1(\T)}\\
    &\gtrsim \sum_{k=2^n+1}^N \frac1{k\log^\g (k+1)}\asymp \log^{1-\g} N -\log^{1-\g} 2^n.
\end{split}
\end{equation}

Next, if $\a$ is odd, by~\eqref{eqSum+1}, we derive
\begin{equation*}\label{V_nL11}
  \begin{split}
     \Vert (\eta_{2^m} f_N)^{(\a)}\Vert_{L_1(\T)}&=\left\Vert \eta_{2^m}\(\sum_{k=1}^N \frac{k^\a}{\log^\g(k+1)}\cos kx\)\right\Vert_{L_1(\T)}\\
       &\lesssim \sum_{k=1}^{2^m} \frac{k^{\a-1}}{\log^\g (k+1)}\lesssim \frac{2^{\a m}}{m^\g}.
  \end{split}
\end{equation*}
Similarly, if $\a$ is even, ~\eqref{eqSum+2} implies that
\begin{equation*}\label{V_nL11+}
  \begin{split}
     \Vert (\eta_{2^m} f_N)^{(\a)}\Vert_{L_1(\T)}\lesssim \frac{2^{\a m}}{m^\g}.
  \end{split}
\end{equation*}
Thus, for all $\a\in \N$, we have
\begin{equation}\label{V_NLLLL}
\begin{split}
    \sum_{m=n}^\infty \(2^{-\a m} \Vert (\eta_{2^m} f_N)^{(\a)}\Vert_{L_1(\T)}\)^q
  &\lesssim \sum_{m=n}^{[\log N]}\frac{1}{m^{\g q}}+\sum_{m=[\log N]+1}^\infty 2^{-\a q m}\frac{N^{\a q}}{(\log N)^{\g q}}\\
  &\lesssim\frac{\log N}{(\log N)^{\g q}}.
\end{split}
\end{equation}

Combining~\eqref{eqE_nL1}  and~\eqref{V_NLLLL}, it is easy to see that inequality~\eqref{Opt2LL2-} is not valid for $f=f_N$ with sufficiently large $N$.

%
\end{proof}

\subsection{Optimality of the left-hand inequality in~\eqref{optimal} for $p=1$ and $p=\infty$}

In this subsection, we show that the left-hand inequality in~\eqref{optimal} cannot be improved in general. In particular, for $p=1$ or $p=\infty$, the following inequality is not valid for any $q>0$
 \begin{equation}\label{Opt2LL2}
    \(\sum_{k=n+1}^\infty 2^{-q \a k}\Vert (\eta_{2^k}f)^{(\a)}\Vert_{L_p(\T)}^q\)^\frac1q\le C\omega_\a(f,2^{-n})_{L_{p}(\T)}.
  \end{equation}

\begin{theorem}\label{Opt2LL}
  Let $p=1$ or $\infty$ and $\a\in \N$. Then for any $q>0$ there exists a function $f\in L_p(\T)$ such that inequality
(\ref{Opt2LL2})
   is not valid with a constant $C$ independent of $n$ and $f$.
\end{theorem}

\begin{proof}
Let $p=\infty$. We take
$$
f(x)=\sum_{m=1}^\infty a_m \sin mx,\quad a_m=\frac{1}{m \log^\gamma (m+1)},\quad \g>1.
$$
Since $a_m\searrow 0$ and $ma_m\to 0$, we have $f\in C(\T)$, see, e.g.,~\cite[Ch.~V]{Z}.

By~\cite{Ti08}, we get
\begin{equation}\label{eqBestOpt}
E_n(f)_{L_\infty(\T)}\asymp \max_{\nu\ge 1}\nu a_{\nu+n} \asymp \max_{\nu\ge 1}\frac{\nu}{(\nu+n)\log^\g (\nu+n+1)}\asymp \frac1{\log^\g n}.
\end{equation}

Next,
$$
\Vert (\eta_{2^k}f)^{(\a)}\Vert_{L_\infty(\T)}=\left\Vert \eta_{2^k}\(\sum_{m=1}^\infty \frac{m^{\a-1}\cos(mx+\a\pi)}{\log^\g (m+1)}\)\right\Vert_{L_\infty(\T)}.
$$
If $\a$ is even, we obviously have
\begin{equation}\label{oddA}
  \Vert (\eta_{2^k}f)^{(\a)}\Vert_{L_\infty(\T)}\asymp \sum_{m=1}^{2^k} \eta\(\frac{m}{2^k}\) \frac{m^{\a-1}}{\log^\g (m+1)}\asymp\frac{2^{\a k}}{k^\g}.
\end{equation}
For odd $\a$, using Bernstein's inequality, we derive
\begin{equation}\label{evenA}
  \begin{split}
      \Vert (\eta_{2^k}f)^{(\a)}\Vert_{L_\infty(\T)}&\ge \frac{1}{2^k} \Vert (\eta_{2^k}f)^{(\a+1)}\Vert_{L_\infty(\T)}\\
      &=\frac 1{2^k} \left\Vert \eta_{2^k}\(\sum_{m=1}^\infty \frac{m^{\a}\cos mx}{\log^\g (m+1)}\)\right\Vert_{L_\infty(\T)}\asymp\frac{2^{\a k}}{k^\g}.
   \end{split}
\end{equation}

Due to \eqref{eqBestOpt}, \eqref{oddA}, and~\eqref{evenA}, and using the realization result, we have
\begin{equation*}\label{eqMOdLem}
  \omega_\a(f,2^{-n})_{L_{\infty}(\T)}\asymp \frac{1}{n^\g}.
\end{equation*}
At the same time, by \eqref{oddA} and~\eqref{evenA}, we derive
\begin{equation*}\label{eqValsumnorm}
  \(\sum_{k=n+1}^\infty 2^{-q \a k}\Vert (\eta_{2^k}f)^{(\a)}\Vert_{L_\infty(\T)}^q\)^\frac1q\asymp \frac{n^\frac1q}{n^\g}.
\end{equation*}
The last two formula imply that inequality~\eqref{Opt2LL2} is not valid in the case $p=\infty$.

Now, let us consider the case $p=1$. We put
$$
f(x)=\sum_{m=1}^\infty {a_m \cos mx},\quad
a_m=\frac{1}{\log^\gamma (m+1)},
\qquad
\g>1.
$$
Since $a_m\searrow 0$ and $\Delta^2 a_m\ge 0$, we have $f\in L_1(\T)$, see, e.g.,~\cite[Ch.~V]{Z}.

Recall that if  a convex sequence $\{a_m\}$ is the sequence of cosine Fourier coefficients of an even function $f\in L_1(\T)$, then applying   Theorem~1 from~\cite{al}, we have
\begin{equation}\label{eqmodoptL1}
  \w_\a(f,2^{-n})_{L_1(\T)}\lesssim \frac1{2^{\a n}}\sum_{m=1}^{2^{n}} {m^{\a-1}}{a_m}\lesssim \frac1{n^\g}.
\end{equation}
Next, since for any $g\in L_1(\T)$ and $k\in \N$, one has $\Vert g\Vert_{L_1(\T)} \ge 2\pi|\widehat{g}(2^k)|$, it follows that
$$
\Vert (\eta_{2^k} f)^{(\a)}\Vert_{L_1(\T)}=\left\Vert \eta_{2^k}\(\sum_{m=1}^\infty \frac{m^{\a}\cos(mx+\a\pi)}{\log^\g (m+1)}\)\right\Vert_{L_1(\T)} \gtrsim \frac{2^{\a k}}{k^\g}
$$
and, therefore,
\begin{equation}\label{eqValsumnormL1}
  \(\sum_{k=n+1}^\infty 2^{-q \a k}\Vert (\eta_{2^k}f)^{(\a)}\Vert_{L_1(\T)}^q\)^\frac1q\gtrsim \frac{n^\frac1q}{n^\g}.
\end{equation}
Finally, combining~\eqref{eqmodoptL1} and~\eqref{eqValsumnormL1}, we obtain contradiction to~\eqref{Opt2LL2}.
\end{proof}

\bigskip
\bigskip


                         \end{document}